\documentclass{amsart}

\usepackage[latin1]{inputenc}
\usepackage{amsfonts}
\usepackage{amsmath}
\usepackage{amsthm}
\usepackage{amssymb}
\usepackage{latexsym}
\usepackage{enumerate}

\newtheorem{theorem}{Theorem}[section]
\newtheorem{lemma}[theorem]{Lemma}
\newtheorem{proposition}[theorem]{Proposition}
\newtheorem{corollary}[theorem]{Corollary}

\theoremstyle{definition}
\newtheorem{definition}[theorem]{Definition}

\numberwithin{equation}{section}

\begin{document}

\newcommand{\cc}{\mathfrak{c}}
\newcommand{\N}{\mathbb{N}}
\newcommand{\PP}{\mathbb{P}}
\newcommand{\QQ}{\mathbb{Q}}
\newcommand{\forces}{\Vdash}
\newcommand{\R}{\mathbb{R}}
\newcommand{\LL}{\mathbb{L}}
\newcommand{\MM}{\mathbb{M}}

\title{On constructions with $2$-cardinals}
\author{Piotr Koszmider}
\address{Institute of Mathematics, Polish Academy of Sciences,
ul. \'Sniadeckich 8,  00-956 Warszawa, Poland}
\email{\texttt{piotr.koszmider@impan.pl}}
\thanks{The  author was partially supported by the National Science Center research grant 
2011/01/B/ST1/00657. }

%
%
%
%
\begin{abstract}
We propose developing the theory of consequences of  morasses
relevant in mathematical applications in the language  alternative to the usual one, replacing
commonly used structures by families of sets originating with Velleman's neat simplified morasses
called $2$-cardinals.
The theory of related trees, gaps, colorings of pairs and forcing notions is reformulated
and sketched from a unifying point of view with the focus on the applicability to constructions of 
mathematical  structures like Boolean algebras, Banach spaces or compact spaces.

The paper is dedicated to the memory of Jim Baumgartner whose
seminal joint paper
 \cite{bs} with Saharon Shelah  provided a critical mass in the theory in question.  

A new  result which we obtain as a side product is the consistency  
of the existence of a function $f:[\lambda^{++}]^2\rightarrow[\lambda^{++}]^{\leq\lambda}$
with the appropriate $\lambda^+$-version of  property $\Delta$
for regular $\lambda\geq\omega$ satisfying $\lambda^{<\lambda}=\lambda$.
\end{abstract}

\maketitle
{}

\markright{}

\section{Introduction}

The notation used is fairly standard\footnote{In particular $|A|$
stands for the cardinality of $A$, $f[A]$ denotes 
the image of $A$ under $f$, $f\restriction A$
denotes the restriction of 
$f$ to $A$. $A\subset B$ means the strict inclusion i.e., $A\not=B$ in that
case.  If $A,B$ are sets of ordinals, then $ordtp(A)$
denotes the order type of $A$ and we write $A<B$ if and only 
if $\alpha<\beta$ for all $\alpha\in A$ and $\beta\in B$.
$ht$ and $rank$ denotes height and rank in well founded families of
sets with respect to the inclusion. $\alpha^{<\beta}$ denotes 
the family of all sequences of elements
from $\alpha$ of length less then $\beta$. If $\kappa$
and $\lambda$ are cardinals, then  $\wp_\kappa(\lambda)=\{
X\subseteq\lambda:\ |X|<\kappa\}$.}, for unexplained symbols
and notions see \cite{kunen} or \cite{jech}. If $\mu\subseteq \wp_\kappa(\lambda)$
and $X\subseteq \lambda$, then
$\mu| X=\{Y\in\mu:\ Y\subset X\}$. If $X$ and $Y$ are sets of ordinals
of the same order type, then $f_{YX}$ denotes the unique order preserving bijection
from $X$ onto $Y$.
By 
$$X_1*X_2$$
we mean $X_1\cup X_2$ if $X_1, X_2$ are two sets of ordinals of the same order type
and $X_1\cap X_2<X_1\setminus X_2<X_2\setminus X_1$. Otherwise $X_1*X_2$ is undefined.
\vskip 13pt

\begin{definition}[\cite{vellemandiam}, \cite{vellemanjsl}]\label{morass} Let $\kappa$ be a 
regular cardinal. A  $(\kappa,\kappa^+)$-cardinal\footnote{Formally,  in the original terminology of 
\cite{vellemanjsl} and \cite{vellemandiam} a  $(\kappa,\kappa^+)$-cardinal
is a neat simplified $(\kappa, 1)$-morass, however in many following papers e.g., \cite{morganforcing},
\cite{koepke}, \cite{irrgang}  a $(\kappa, 1)$-morass is what formally Velleman called
an expanded neat simplified morass.
This shift towards the expanded version (already present in the above papers of Velleman)
is justified  by the fact that the above authors do all the calculations with the expanded versions 
i.e.,   use maps rather than sets.}
is a family $\mu\subseteq\wp_\kappa(\kappa^+)$ which satisfies the following
conditions:
\begin{enumerate}
\item  $\mu$ is well-founded with respect to inclusion,

\item $\mu$ is {\sl locally small} i.e. $ |(\mu| X) |
<\kappa $ for all $X\in \mu$,

\item $\mu$ is {\sl homogenous} i.e. if $X,Y\in\mu,\ rank(X)=rank(Y) 
$, then $X,Y$ have the same order type and
$\mu| Y=\{f_{YX}[Z]:\ Z\in\mu| X\}$,

\item $\mu$ is {\sl directed} i.e.,  for every $X,Y\in\mu$ there exists $Z\in\mu$
such that $X,Y\subseteq Z$,
\item $\mu$ is {\sl locally almost directed}, i.e., for every $X\in\mu$ either
\begin{enumerate}[(a)]
\item  $\mu| X$ is directed or
\item  there are $X_1,X_2\in\mu$ of the same rank such that  
$$X=X_1*X_2 \ \hbox{and}\  \mu| X=
(\mu| X_1)\cup(\mu| X_2)\cup\{X_1,X_2\}$$
 \end{enumerate}
\item $\mu$ {\sl covers} $\kappa^+$ i.e., $\bigcup\mu=\kappa^+$.
\item $\mu$ is {\sl neat},  that is  for every element $X$ of $\mu$ of nonzero rank we have
$$X=\bigcup(\mu|X).$$ 
\end{enumerate}
\end{definition}
\vskip 13pt

The terminology proposed here is a suggested consequence of the main point of the paper which is that the above representation
of $(\kappa,1)$-morass allows to shift the language of the theory (proofs, lemmas, theorems) into
a language compatible with the  part of set theory applicable in classical mathematical
fields  (forcing, partitions, transfinite recursion
rather than the spirit of the fine structure of $L$, inner models etc.). 
 The allusion in the terminology is that the above representation
is as liberating, compared to the usual morass language, as von Neumann's ordinals compared to
Cantor's theory of well-orders and embeddings among them. 
 We propose here rewriting all the standard calculations
in the language of Definition \ref{morass} and claim that we obtain 
a quite transparent and usable theory when it starts living its own life without the reference to
the old body of arguments.

At first sight one can doubt if it matters to talk about some sets in $\kappa^+$ of cardinalities
less than $\kappa$, their intersections and unions,
rather than mappings from ordinals less than $\kappa$ into $\kappa^+$ and appropriate compositions.
We feel that, however, the
 degree of the challenge of building this complex theory in a language that carries unnecessary information
may be well expressed in the word ``simplified morass"  and that there is a substantial progress
if one moves to families of sets and tries to settle all milestones and references in  a new language
which is more compatible with the language of places where it is needed:  forcing with
models as side conditions, constructions of classical mathematical structures or partitions.
For example, these simplifications in many cases provide explicit definitions of the required
objects instead of recursive ones, even in highly complex cases as the Hausdorff gaps
or colorings of pairs similar to the $\rho$-function.

In any case, we hope  that this text could serve as a relatively painless introduction to
applications of morasses,  as its diverse circulated unpublished versions functioned 
this way in the last two decades under the name {\it Etude in simplified morasses}.

The notation $(\kappa, \kappa^+)$-cardinal suggest the possibility of
using different pairs of cardinals or longer sequences of them. Indeed one could
consider a $(\kappa,\lambda)$-semimorasses of \cite{semi} as a $(\kappa,\lambda)$-cardinals.
In the case of $\lambda>\kappa^+$, as in \cite{semi} one needs to change the definition
of $X_1*X_2$, replacing the condition $X_1\cap X_2<X_2\setminus X_2<X_2\setminus X_1$
by $f_{X_2X_1}\restriction X_1\cap X_2=Id_{X_1\cap X_2}$.

The structure of the paper is the following.
In Section 2 we develop elementary properties of $2$-cardinals as
families of sets.
In Section 3 we focus on recursive constructions along $2$-cardinals
in the analogy to the usual transfinite recursion along ordinals. 
In fact, the heart of the philosophy of $(\kappa,\kappa^+)$-cardinals is to view $\kappa^+$
as built from fragments of sizes less than $\kappa$ so that a recursive construction
of a structure of size $\kappa^+$ does not have to deal with the case of an
intermediate construction having size $\kappa$.
As the  main example  we propose  a direct construction of a $\kappa$-thin tall
Boolean algebra due to Koepke and Martinez \cite{koepke}, instead of using the
 morass version of Martin's axiom developed by Velleman in \cite{vellemandiam}
and \cite{vellemanjsl}. In a sense the claim of this section is that we can do very well
without this version of Martin's axiom, if we represent appropriately the intermediate structures.

Certainly the proof of the equivalence of this version of Martin's axiom and the existence of
simplified morasses played historically a very important  role. However, looking backwards,
in practice it seems that either one can
do a transfinite recursion along a $(\kappa,\kappa^+)$-cardinal with a nicely represented structures
or there is a need of additional ad hoc properties, which may be reduced to
increasingly complex versions of morasses with built-in diamond as in \cite{vellemandiam} and then
the actual forcing approach turns out to be more economic at least in applications
(e.g. \cite{isaac}, \cite{init}, \cite{big}, \cite{grande}), where what matters most is the consistency and not
necessarily holding in the constructible universe.

Thus, Velleman's version of Martin's axiom equivalent to a morass is not discussed nor proposed
as a convenient tool here.
We refer the interested reader to \cite{vellemandiam}
and \cite{vellemanjsl}. Perhaps the initial idea of formulating morasses like in the Definition \ref{morass}
was motivated by this version of Martin's axiom. As it turned out not to be often used (the same fate
was met by another attempt in this direction \cite{shelahstanley}), the language
shifted in the direction of expanded simplified morasses.
On the other hand, one 
should remember that proving that a theorem follows 
from the existence of a morass or a $2$-cardinal means that an inaccessible
cardinal is necessary to obtain the consistency
of the negation of the theorem. This cannot be said about
consistency proofs which use the method of forcing.\par

In Section 4 we give canonical definitions of several classical objects from
a $(\kappa,\kappa^+)$-cardinal, we propose a couple of types of Kurepa trees and 
generalizations of Hausdorff gaps. The simplicity of these definitions and the
proofs of the properties,  especially in the context of  the importance of theses structures, shows
 that the language of $(\kappa,\kappa^+)$-cardinal indeed clears the working environment.
In this section we also  prove that a stationary $(\kappa,\kappa^+)$-cardinal is
a stationary subset of
$\wp_\kappa(\kappa^+)$ which does not reflect to $\wp_\kappa(A)$
for any proper subset $A\subset \kappa^+$. We also review
some relevant literature concerning the combinatorial phenomena displayed by
the above objects, which often is referred to as  noncompactness, nonreflections or gaps.

In Section 5 we develop a theory of $\rho$-function type coloring which
can be canonically defined from a $(\kappa,\kappa^+)$-cardinal.
It was C. Morgan  who first saw such a possibility  in \cite{morganforcing}.
Our approach allows to obtain a function with $\Delta$-property for
$\kappa>\omega_1$ generalizing previous results.

In section 6 we review possible applications of
$2$-cardinals for building forcing notions.

Finally in Section 7 we mention attempts of transforming higher gap morasses
into a tool manageable in applications.

The morasses were introduced by
R. Jensen (see \cite{devlin2}). It is beyond the scope of this paper
to give a historical review of their profound impact.
Using the results of Velleman (\cite{vellemanjsl}, \cite{vellemandiam}) we can conclude
from Jensen's theory that in $L$ there exists
a $(\kappa,\kappa^+)$-cardinal for every regular  uncountable $\kappa$ and if
there is no $(\kappa,\kappa^+)$-cardinal for such a $\kappa$, then it
is inaccessible in $L$. $2$-cardinals can also be easily added by a nice forcing like in
\cite{semi}. In the case of $\kappa=\omega$ Velleman's morasses exist in ZFC (\cite{vellemanzfc})
and do not have classical counterpart in Jensen's theory. In Section 4 we give an explicit
definition of a Hausdorff gap from such a 2-cardinal.
In the language
of \cite{stevoacta}, a $(\kappa,\kappa^+)$-cardinal can be considered a stepping-up tool, it 
enables us to step-up  properties of $\kappa$,
obtained by the usual induction, to $\kappa^+$, since 
the initial fragments of the constructions are 
of sizes less than $\kappa$.
In the above sense every  well-founded directed
set of size $\kappa^+$
with initial fragments of sizes less than $\kappa$ is
a stepping-up tool. Additional strength and the essence of
a $2$-cardinal as well as other nontrivial stepping-up 
frameworks is hidden in  coherence properties
of the framework.\par

We will focus on the possibilities of using the language of Definition \ref{morass}
and so we have to omit most of the  comments on the complicated network of results
concerning the consistency strengths of the combinatorial principles which appear in this paper as well
as the comments on the relations of the constructions to the fine structure of $L$.

\vskip 26pt
\section{Elementary properties}

\begin{lemma}[The coherence lemma (2.4. \cite{vellemanzfc})]\label{coherence}
 Let $\kappa$ be a regular cardinal and  $\mu $ be a 
$(\kappa,\kappa^+) $-cardinal. Let $X,Y\in\mu$ be of the same rank and let
$\alpha\in X\cap Y$, then $$X\cap\alpha=Y\cap\alpha.$$
\end{lemma}
\begin{proof}
By induction on $rank(Z)$ such that $X,Y\subseteq Z\in\mu$ which exists by the
directedness \ref{morass} (4).
If (5a) of \ref{morass} holds for $\mu|Z$, then 
we are immediately done by the inductive hypothesis.

If (5b) of \ref{morass} holds, we have $Z=Z_1*Z_2$, and say $X\subseteq Z_1,\
Y\subseteq Z_2$, (otherwise we are done by inductive hypothesis).
By $Z=Z_1*Z_2$ we have that $f_{Z_2 Z_1}\restriction (Z_1\cap Z_2)= Id_{Z_1\cap Z_2}$
and $Z_1\cap\alpha'=Z_2\cap\alpha'$ for any $\alpha'\in Z_1\cap Z_2$
in particular for $\alpha\in Z_1\cap Z_2$,
since $\alpha\in X\cap Y$. By the homogeneity \ref{morass} (3)
$rank(f_{Z_2 Z_1}[X])=rank(X)=rank(Y)$.
We know also that $\alpha\in f_{Z_2 Z_1}[X]$, since
 $f_{Z_2 Z_1}\restriction Z_1\cap Z_2$ is the identity.
Now, by inductive hypothesis for $Z_2$, we obtain that 
$$f_{Z_2 Z_1}[X]\cap\alpha=Y\cap\alpha,$$
 but again 
since $f_{Z_2 Z_1}\restriction Z_1\cap Z_2=Id_{Z_1\cap Z_2}$, we have
$f_{Z_2 Z_1}[X]\cap\alpha=X\cap\alpha$, so $Y\cap\alpha=
X\cap\alpha$ as required. 
\end{proof}

Using the coherence lemma we can conclude the lemma below even in the case
when $X*Y\not\in \mu$.

\begin{lemma}\label{corcoherence}
Suppose that $\kappa$ is a regular cardinal and $\mu$ is a 
$(\kappa, \kappa^+)$-cardinal.  Let $X$ and $Y$ be elements of
$\mu$ of the same rank, then $f_{XY}\restriction (X\cap Y)=Id_{X\cap Y}$.
\end{lemma}

\begin{lemma}[The density lemma (2.7. \cite{vellemanjsl})]\label{density}
 Suppose that $\kappa$ is a regular cardinal and $\mu$ is a 
$(\kappa, \kappa^+)$-cardinal.  Then the following conditions are
satisfied:
\begin{enumerate}

\item If $X\in\mu$
$$\{rank(Z): Z\in\mu, \ X\subseteq Z \}=[rank(X), ht(\mu)).$$

\item If $X\subseteq Y$ are two elements of $\mu$, then
$$\{rank(Z): Z\in\mu, \ X\subseteq Z \subseteq Y\}=[rank(X), rank(Y)].$$

\end{enumerate}
\end{lemma}

\begin{proof} To prove (1) fix $X\in\mu$
take $rank(X)\leq\alpha<ht(\mu)$ and take $Y\in\mu$ of minimal rank such that $X\subseteq Y$ and
$rank(Y)\geq\alpha$. There is such a $Y$ by the directedness of $\mu$ \ref{morass} (4).

 If $rank(Y)=\alpha$ we are done. We will prove that 
 $rank(Y)>\alpha$ gives rise to a contradiction. We 
 apply the local almost directedness \ref{morass} (5) of $\mu$  to $Y$.
The minimality of the rank of $Y$ implies that $\mu|Y$ cannot be directed, so
 there are $Y_1, Y_2$ such that
$Y=Y_1*Y_2$, then $rank(Y)=rank(Y_i)+1>\alpha$, so
 $rank(Y_i)\geq\alpha$ and $X\in\mu\restriction Y_1$ 
 or $X\in\mu\restriction Y_2$. This 
 contradicts the minimality of the rank of $Y$
 and completes the proof of part (1).\par
 \noindent For (2) fix  $X,Y\in\mu$ and $\alpha<ht(\mu)$ such that $X
 \subseteq Y$ and  $rank(X)<\alpha<rank(Y)$.
 Using the part (1), find $Z_1,Z_2\in \mu$ such that
 $X\subseteq Z_1\subseteq Z_2$ and
 $rank(Z_1)=\alpha$ and $rank(Z_2)=rank(Y)$. Consider
 $f_{Y Z_2}$, we get
 $f_{Y Z_2}[X]\subseteq f_{Y Z_2}[Z_1]$ and 
 $f_{Y Z_2}[Z_1]\in\mu$ and $rank(f_{Y Z_2}[Z_1])=\alpha$.
 It is enough to prove that $f_{Y Z_2}[X]=X$, i.e.,
 $f_{Y Z_2}(\alpha)=\alpha$ for all $\alpha$'s
 in $X$. But as $X\subseteq Z_2, Y$, if $\alpha\in X$ and $\alpha\in Z_2\cap Y$,
 the coherence lemma \ref{coherence} implies that $ordtp(\alpha\cap Z_2)=
 ordtp(\alpha\cap Y)$.  As $f_{Y Z_2}$ is order
 preserving, we get that $f_{Y Z_2}(\alpha)=\alpha$. 
\end{proof}

\begin{lemma}[The localization lemma]\label{localization}
Suppose that $\kappa$ is a regular cardinal and $\mu$ is a 
$(\kappa, \kappa^+)$-cardinal. 
Suppose that $F\subseteq \kappa^+$ is a finite set such that there
is $X\in\mu$ with $rank(X)=\eta$ and $F\subseteq X$. If $Y\in \mu$ contains $F$ and is
of $rank(Y)\geq\eta$, then there is $X'\in (\mu|Y)\cup\{Y\}$ such that
$$F\subseteq X'\ \hbox{and}\  rank(X')=\eta.$$
In particular $X\cap \max(F)=X'\cap \max(F)\subseteq Y$.
\end{lemma}
\begin{proof}
If $rank(Y)=\eta$, then $X'=Y$ works. So we may assume that $rank(Y)>\eta.$
By the density lemma \ref{density} there is $Y'\in \mu$ such that
$rank(Y')=rank(Y)$ and $X\subset Y'$. Now use the homogeneity
 \ref{morass} (3) of $\mu$  to note that $X'=f_{YY'}[X]\in \mu|Y$.
By the coherence lemma $Y\cap \max(F)=Y'\cap\max(F)$, so
$f_{YY'}$ as an order preserving map is the identity on $Y\cap \max(F)=Y'\cap\max(F)$,
in particular $F\subseteq X'$ and $X\cap \max(F)=X'\cap\max(F)$.
\end{proof}

\begin{lemma}\label{zerorank}
Let $\kappa$ be a regular cardinal and $\mu$ be a $(\kappa, \kappa^+)$-cardinal.
Every element $\alpha\in\kappa^+$ is in some $X\in\mu$ of rank zero.
\end{lemma}
\begin{proof}
Let $X$ be of minimal rank such that $\alpha\in X$, which exists by (6) of \ref{morass}.
By the neatness (7) of \ref{morass} the rank of $X$ must be zero.
\end{proof}

\begin{definition}\label{musequencedef} 
Let $\kappa$ be a regular cardinal and $\mu$ be a $(\kappa, \kappa^+)$-cardinal.
 Let $\alpha\in\kappa^+$. The sequence 
$(\mu_\xi(\alpha))_{\xi<ht(\mu)}$ is called the $\mu$-sequence at  $\alpha$ if
and only if  for all $\xi<ht(\mu)$ we have
$$\mu_\xi(\alpha)=X_\xi\cap\alpha,$$
 where $X_\xi\in\mu$ is such that $rank(X_\xi)=\xi,\ 
\alpha\in X_\xi,\  $.
\end{definition}

The fact that $\mu$-sequences are well-defined follows from the coherence lemma \ref{coherence},
the density lemma \ref{density} and Lemma \ref{zerorank}.

\begin{lemma}\label{musequencecoherence}
Suppose that $\kappa$ is a regular cardinal and $\mu$ is a $(\kappa, \kappa^+)$-cardinal.
Suppose $(\mu_\xi(\alpha))_{\xi<ht(\mu)}$ is a $\mu$-sequence at $\alpha$ and
$\beta\in \mu_\xi(\alpha)$. Then 
$$\mu_\xi(\beta)=\mu_\xi(\alpha)\cap\beta.$$ 
\end{lemma}
\begin{proof}
This is just the coherence lemma \ref{coherence}.
\end{proof}

In  the following lemma we note, among others, that the height of
$\mu$ is $\kappa$ and so the length of
the $\mu$-sequences is $\kappa$, thus they will be denoted
$(\mu_\xi(\alpha))_{\xi<\kappa}$.

\begin{lemma}\label{musequencelemma} Let $\kappa$ be a regular cardinal and $\mu$ be a $(\kappa,\kappa^+)$-cardinal and
let $\alpha\in\kappa^+$.  Then the $\mu$-sequence $(\mu_\xi(\alpha))_{\xi<ht(\mu)}$ 
at $\alpha$ is a continuous non-decreasing sequence 
such that the union of its terms is equal to $\alpha$. 
In particular it is a club subset of $[\alpha]^{<\kappa}$ and so the  height  $ht(\mu)$ of the well-founded set $(\mu, \subseteq)$ is $\kappa$.
\end{lemma}
\begin{proof}

Let $\xi<\xi'<ht(\mu)$.  By
the density lemma \ref{density} there is $Z\in\mu$ of rank $\xi'$
such that $\mu_\xi(\alpha)\cup\{\alpha\}\subseteq Z$ and  then the coherence lemma \ref{coherence}
implies that $Z\cap\alpha=\mu_{\xi'}(\alpha)$, so $\mu_\xi(\alpha)\subseteq \mu_{\xi'}(\alpha)$.
The directedness and covering of $\kappa^+$ imply that the union 
is equal to $\alpha$. The neatness and the directedness \ref{morass} imply the continuity. 

To prove that   $(\mu_\xi(\alpha))_{\xi<ht(\mu)}$ is unbounded  
in  $[\alpha]^{<\kappa}$ pick any $X=\{\alpha_\eta: \eta<\theta\}\in [\alpha]^{<\kappa}$
for some $\theta<\kappa$. By the first part of the lemma for each $\eta<\theta$
there is $\xi_\eta<\kappa$ such that $\alpha_\eta\in \mu_{\xi_\eta}(\alpha)$.
By the regularity of $\kappa$ there is $\xi<\kappa$ such that $\xi_\eta<\xi$ 
for each $\eta<\theta$. Using the monotonicity of the $\mu$-sequence from the first part
of the proof we conclude that $X\subseteq \mu_\xi(\alpha)$ as required.

To evaluate the height of $\mu$ note that since all elements of $\mu$ are of cardinalities less than
 $\kappa$,  $\kappa$ is regular and the $\mu$-sequence  at $\kappa$
 covers $\kappa$, there must be at least $\kappa$ ranks, that is
$ht(\mu)\geq \kappa$. Since $\mu$ is locally small (\ref{morass} (2)) we have that $ht(\mu)\leq \kappa$.

\end{proof}

Thus for every $\alpha\in \kappa^+$ the $\mu$-sequence at $\alpha$ provides a decomposition
of $\alpha$ as a nondecreasing continuous chain in type $\kappa$ which covers $\alpha$.
Moreover by \ref{musequencecoherence} these chains for different $\alpha$s cohere.

\begin{lemma}\label{musequencecofcof}
 Let $\kappa$ be a regular cardinal and 
 $\mu$ be a  $(\kappa,\kappa^+)$-cardinal. Let $\alpha\in\kappa^+$
and $\delta$ be a limit ordinal.
If $X\in \mu$ is of rank less than $\delta$, then there is $\delta'<\delta$ such that
$$X\cap\mu_{\delta}(\alpha)\subseteq \mu_{\delta'}(\alpha).$$
\end{lemma}
\begin{proof}
Let  $Y\in\mu$ be of rank $\delta$ such that
$\alpha\in Y$. Using the density lemma \ref{density} find $Z\in\mu$ of rank $\delta$ such that
$X\subseteq Z$.

Consider $X'=f_{YZ}[X]\in \mu|Y$. As $X\cap\mu_{\delta}(\alpha)\subseteq X, Y, Z$, by
\ref{corcoherence} we have that $X\cap\mu_{\delta}(\alpha)\subseteq X'$. By the 
almost directedness and by the neatness
there is $X''\in\mu|Y$ of rank $\delta'<\delta$ such that $X'\cup\{\alpha\}\subseteq X''$. So
$$X\cap\mu_{\delta}(\alpha)\subseteq X''\cap\alpha=\mu_{\delta'}(\alpha),$$
as required.
\end{proof}

\begin{proposition}\label{cofinality} Let $\kappa$ be a regular cardinal and $\mu$ be a $(\kappa,\kappa^+)$-cardinal.
Then $\mu$ is
cofinal in $([\kappa^+]^{<\kappa},\subseteq)$.
\end{proposition}
\begin{proof}
Let $X\in[\kappa^+]^{<\kappa}$, choose $\alpha\in\kappa^+$ 
of cofinality $\kappa$ such
that $\sup(X)<\alpha$, 
consider the $\mu$-sequence at $\alpha$. By Lemma \ref{musequencelemma} 
there is an element of it which 
 includes $X$. 
\end{proof}

\noindent In particular a $(\kappa,\kappa^+)$-cardinal $\mu$  is a 
cofinal   family  in $\wp_\kappa
(\kappa^+)$ which is the union of at most $\kappa$ many
subfamilies ${\mu}_\alpha$ for $\alpha<\kappa$ 
(i.e, ${\mu}_\alpha$ consists of elements of 
rank $\alpha$)
 such that for every
two $X, X'\in {\mu}_\alpha$ such that
$\sup(X)\leq \sup(X')$ we have $X\cap X'<X\setminus X'<X'\setminus X$ 
by the coherence lemma \ref{coherence}.\par
Note that the families ${\mu}_\alpha$ cannot 
be $\Delta$-systems, i.e, have the property that 
there is $\Delta_\alpha\in\wp_\kappa(\kappa^+)$ such that
for each $X, X'\in {\mu}_\alpha$ we have 
$\Delta_\alpha=X\cap X'$. Note also that there is no
family ${\mu}\subseteq\wp_\kappa(\lambda)$, 
satisfying  definition \ref{morass} for $\lambda>\kappa^+$.
To see this, suppose that $\lambda>\kappa^+$ and 
consider the $\mu$-sequence at $\kappa^+$, as defined in definition
\ref{musequencedef}, by Lemma \ref{musequencelemma} it covers $\kappa^+$, but 
$\kappa^+$ is regular, so it cannot be covered
by this sequence. Now let us make some elementary observation concerning
the interaction of $2$-cardinals and elementary submodels.

\begin{lemma}\label{musequencemodel} Let $\kappa$ be a regular cardinal and $\mu$ be a 
 $(\kappa,\kappa^+)$-cardinal. Suppose that $M\prec H(\kappa^{++})$ is an elementary
submodel of cardinality less than $\kappa$ which contains $\mu$
and such that $\delta=M\cap\kappa\in\kappa$. Then 
\begin{enumerate}
\item For every $\alpha\in M\cap\kappa^+$
we have $$M\cap\alpha=\mu_\delta(\alpha).$$ 
\item If $M\cap \kappa^+\in \mu$, then $rank(M\cap\kappa^+)=\delta$.
\end{enumerate}
\end{lemma}
\begin{proof} (1)
If $\beta\in M\cap \alpha$, then by the covering
and directedness \ref{morass} and by the elementarity there is $\xi\in M\cap\kappa=\delta$ such that
$\beta\in \mu_{\xi}(\alpha)$, so by the fact that $(\mu_\xi(\alpha))_{\xi<\kappa}$ is
nondecreasing we get that $\beta\in \mu_\delta(\alpha)$. 

As $\delta=M\cap\kappa\in\kappa$ must be a limit ordinal,
using the directedness and the neatness of $\mu$, if $\beta\in \mu_\delta(\alpha)$, then
there is $\delta'<\delta$ such that $\beta\in \mu_{\delta'}(\alpha)$.
So, as $\delta', \alpha, \mu\in M$, we get that $\beta\in M$.

(2) If $rank(M\cap\kappa^+)$ were less than $\delta$, 
by the elementarity, we would have $ordtp(M\cap\kappa^+)\in M$
which is impossible. If $\delta$ were less than $rank(M\cap\kappa^+)$,
there would exist $X\in \mu|(M\cap\kappa^+)$ of rank bigger than $\delta$,
say of successor rank $rank(M\cap\kappa^+)>\delta'>\delta$ of the form $X_1*X_2$
and so there would exist $\alpha\in M$ such that
$\mu_\delta(\alpha)\not=\mu_{\delta'}(\alpha)\subseteq M\cap \alpha$
contradicting (1).

\end{proof}

\begin{lemma} Let $\kappa$
be a regular cardinal  and
let $\mu$ be a 
 $(\kappa,\kappa^+)$-cardinal. Suppose that $M\prec H(\kappa^{++})$ is an elementary
submodel such that $M\cap\kappa^+$ has cardinality $\kappa$ and  contains $\mu$.
Then $M\cap\kappa$ is unbounded in $\kappa$.
In particular Chang's Conjecture fails at $\kappa$ 
\end{lemma}
\begin{proof}
Let $\theta<\kappa$. Using the fact that $M\cap\kappa^+$ has cardinality $\kappa$
find $\alpha\in M\cap\kappa^+$ such that the order type of $M\cap\alpha$ is bigger than
the order type of elements of $\mu$ of rank $\theta$. This means that there is
$\xi\in M\cap\alpha$ such that $\xi\not\in \mu_\theta(\alpha)$. 
Hence some ordinal $\theta'<\kappa$ such that $\xi\in \mu_{\theta'}(\alpha)$ is definable in $M$ and
bigger than $\theta$ implying that $\theta$ is not a bound for $M\cap\kappa$ which completes
the proof.

\end{proof}

If $\mu$ is stationary in $\wp_\kappa(\kappa^+)$, then we have elementary submodels $M$ such that
$M\cap \kappa^+$ are in $\mu$, in this case (2) of \ref{musequencemodel} is not vacuous.
We may moreover
require that $\mu$ is
a stationary coding set (see \cite{zwicker}). By definition this means that
$\mu$ is stationary subset of $\wp_\kappa(\kappa^+)$
and that there is a one-to-one function $c:\mu\rightarrow\kappa^{+}$
such that
$$\forall X,Y\in \mu\ \ 
X\subset Y\ \Rightarrow c(X)\in Y.$$
\noindent The forcing proof of the existence of  neat morasses
which are stationary coding sets which is based on
a proof of Velleman from \cite{vellemanjsl} can be  obtained 
from the corresponding proof for semimorasses in \cite{semi} (Theorem 3, Section 2). 
\noindent Let us note two simple facts about stationary coding sets:\par
\vskip 13pt
\begin{proposition}Suppose that $\kappa$ is a regular cardinal and that a $(\kappa,\kappa^+)$-cardinal
$\mu\subseteq\wp_\kappa(\kappa^{+})$ is a stationary coding set
and $\mu\in M \prec H(\kappa^{++})$,
 $|M|<\kappa$, $M\cap \kappa^+\in \mu$.
If $X\in\mu$ and $X\subset M$, then $X\in M$.
\end{proposition}
\begin{proof} Suppose $X\in\mu$ and $X\subset M$.
As $\mu\in M \prec H({\kappa^{++}})$, we have that
$M$ {\it thinks} that $\mu$ is a stationary coding set,
so there is $c:\mu\rightarrow\kappa^+$ witnessing this
fact in $M$. In particular $\alpha=c(X)\in M\cap\kappa^{+}$,
so $X= c^{-1}(\alpha)$ is in $M$, as
required.
\end{proof}

\noindent The fact below is crucial in our method of forcing with side
conditions in morasses which we introduced in \cite{unbounded} and which
is outlined in the context of this paper in Section 6.

\begin{lemma}\label{stationarycoding}  Suppose that $\kappa$ is  a regular cardinal, a $(\kappa,\kappa^+)$-cardinal 
$\mu\subseteq\wp_\kappa(\kappa^+)$ is a stationary coding set
and $\mu\in M \prec H(\kappa^{++})$,
 $|M|<\kappa$, $M\cap \kappa^+=X_0\in \mu$.  Let
$Y\in \mu$, $rank(Y)<M\cap\kappa=\delta$. Then there is
$Z(Y)\in M\cap\mu$
 such that\par
1) $Y\cap X_0\subseteq Z(Y)$.\par
2) $rank(Z(Y))=rank(Y)$.\par
\end{lemma}
\begin{proof}
Use the density lemma \ref{density} to find $X\in\mu$ such that
$Y\subseteq X$ and $rank(X)=rank(X_0)$. Now use the homogeneity
\ref{morass} to construct $Z(Y)=f_{X_0X}[Y]$ satisfying (2).
By the previous proposition $Z(Y)\in M$.
To prove (1)  note that $Y\cap X_0\subseteq X\cap X_0$ and $f_{X_0X}$
is the identity on $X\cap X_0$ by \ref{corcoherence}.
\end{proof}

Note  that by the coherence lemma \ref{coherence}, it follows that in the above lemma 
$Z(Y)$ is an end-extension of $X_0\cap Y$.
\vskip 26pt

\section{Recursive constructions}

In this section we give an example of a recursive construction where the recursion is carried
out along a $2$-cardinal instead of the usual one-dimensional cardinal.
Instead the usual chain $(S_\alpha: \alpha<\kappa)$ where
$S_\alpha$ is a nice substructure of $S_{\alpha'}$ for $\alpha<\alpha'<\kappa$ we consider
a well-founded directed system  $(S_X: X\in\mu)$ where
$\mu$ is a $(\kappa,\kappa^+)$-cardinal and $S_X$ is a nice substructure of $S_Y$
whenever $X\subseteq Y$ and $X,Y\in\mu$. The well-foundedness allows us to do
a recursive definition of the structures $S_X$. In the case of $X$ of a limit rank, we use the
directedness \ref{morass} and take an appropriate limit of the directed system $(S_Y: Y\in \mu|X)$.
In the case of $X=X_1*X_2$ we may take advantage of the coherence properties of 
a $2$-cardinal if our structures $S_Y$ are nicely related to the order of $Y\in\mu$ inherited from
$\kappa^+$. Namely, we may assume that $X_1$ and $X_2$ are isomorphic (in a sense depending
on the context) and that the isomorphism is the identity on the substructure 
induced by $X_1\cap X_2$, if the structures $S_X$ for $X\in\mu$ involve also substructures
$S_{X\cap\alpha}$ for $\alpha\in X$  which are completely determined by $X\cap \alpha$.
The coherence lemma \ref{coherence} should then  imply that 
$$S_{X_1\cap\alpha_1}=S_{X_2\cap \alpha_2},$$
where $\alpha_i=\min X_i\setminus X_{3-i}$ for $i=1,2$. 
The inductive step may be successful if the existence of such an isomorphism which
is the identity on ``the common part" allows  us to amalgamate the structures $S_{X_1}$ and $S_{X_2}$
into $S_{X_1*X_2}$ maintaining the fact that 
$S_{(X_1*X_2)\cap\alpha}$ for $\alpha\in X_1*X_2$  is  determined by $X\cap \alpha$.
The final structure is obtained as an appropriate limit of $(S_X:X\in\mu)$.

Actually, the above determination of structures
$S_{X\cap\alpha}$  in the construction hints to an explicit and not recursive definition
of the final structure. In many cases described in this paper, we present such explicit definitions
obtained by analyzing the recursive process along a $2$-cardinal (see the next section). 
On the other hand it is like with the usual
linear recursion along an ordinal, the recursion can be so complex that it is more readable to find 
and present the right construction using the recursive definition instead of an explicit one.

One should observe the analogy of the above described constructions with 
forcing the entire structure with substructures $S_X$ for $X\in\mu$ or $X\in \wp_{\kappa}(\kappa^+)$. 
The forcing can be $\kappa$-closed, so we face the problem of proving
that it is $\kappa^+$-c.c. which reduces to appropriate amalgamations.
This analogy, of course, is behind Velleman's or Shelah and Stanley's formulation of morasses 
in the language of a forcing axiom (\cite{vellemanjsl}, \cite{shelahstanley}).

In this section we present a version of the result of Koepke and Martinez
involving superatomic Boolean algebras.
Recall that a superatomic algebra is called $\kappa$-thin tall if and only if
it has height $\kappa^+$ and width $\kappa$ (see a survey paper of
J. Roitman \cite{roitman} for the terminology concerning superatomic Boolean algebras).

\begin{theorem}[\cite{koepke}]\label{thintall}
Suppose that $\kappa$ is a regular cardinal and there exists
a $(\kappa,\kappa^+)$-cardinal. Then there is a $\kappa$-thin tall
superatomic Boolean algebra.
\end{theorem}

When working with partial orders below, by compatibility of two elements $t, s$ 
we mean the forcing compatibility that is the existence of $u\leq t, s$; 
if $u\leq t$ or $t\leq u$ then, we say that $u$ and $t$ are comparable.

\begin{definition}\label{order} Let $\kappa$ be a cardinal and $A\subseteq \kappa\times\kappa^+$.
 We say that
a strict partial order $\ll$ on $A$ is an $A$-order if and only if : 
\begin{enumerate}
\item if $s=\langle\xi, \alpha\rangle$, $t=\langle\xi',\beta\rangle$ 
are distinct and $s\ll t$, then $\alpha<\beta$,
\item every pair $s,t$  of compatible  elements of $A$ has  the infimum,
that is  the set $\{u\in A: u\ll s,t\}$ has the $\ll$-biggest element denoted
by $i_A(s,t)=i(s,t)$.
\end{enumerate}
If $A\subseteq B\subseteq \kappa\times\kappa^+$ and $(A,\ll_A)$ and $(B,\ll_B)$
are $A$ and $B$-orders respectively, then we say that $(A,\ll_A)$ is a good
suborder of $(B,\ll_B)$ whenever it is a suborder and $i_A(s,t)=i_B(s,t)$
for $s,t\in A$.

Moreover we say that $(A, \ll_A)$ is admissible if 
  whenever $\alpha<\beta<\kappa^+$ appear among the second coordinates of elements
of $A$
 and $t=\langle\xi',\beta\rangle\in A$
then $$\{\xi: s=\langle\xi,\alpha\rangle, s\ll_A t\}$$ is infinite.

\end{definition}

The construction of a $\kappa$-thin tall superatomic Boolean algebra can be
easily reduced to an appropriate order by the following:

\begin{proposition}[\cite{koepke}] If there is a $(\kappa,\kappa^+)$-order
which is admissible, then there is
a $\kappa$-thin tall superatomic Boolean algebra.
\end{proposition}

So from this point on we focus on constructing a $(\kappa,\kappa^+)$-order
which is admissible.
To carry out our  recursion we need  a few lemmas.

\begin{lemma}\label{limitorder} Let $\nu\subseteq \wp_\kappa(\kappa^+)$ be
 a directed family and   for every $X\in \nu$ let $\eta_X$ be an ordinal and 
 $(\eta_X\times X, \ll_X)$ be an $\eta_X\times X$-order which is admissible.  Suppose that
if $X, Y$ are elements of $\nu$ with $X\subseteq Y$, then
$\eta_X\leq \eta_Y$ and 
$(\eta_X\times X, \ll_X)$  is a good suborder of $(\eta_Y\times Y, \ll_Y)$. 
 Then putting $Z=\bigcup_{X\in \nu} X$ and
$\eta_Z=\sup_{X\in\nu}\eta_X$ and $\ll_Z=\bigcup_{X\in \nu} \ll_X$ we have that
 $(\eta_Z\times Z, \ll_Z)$ is an $\eta_Z\times Z$-order which is admissible such that
each $(\eta_X\times X, \ll_X)$ is a good suborder of it.

\end{lemma}

\begin{lemma}\label{successororder} Let $\kappa$ be a 
regular cardinal and $\mu$ be a $(\kappa, \kappa^+)$-cardinal.
Suppose that $X\in\mu$ is of successor rank and $X=X_1*X_2$. 
 Suppose that $\eta<\kappa$ and that
$(\eta\times X_1,\ll_{1})$ and $(\eta\times X_2,\ll_{2})$ are
$(\eta, X_1)$-order and $(\eta, X_2)$-order respectively which are admissible and such that
 $f:\eta\times X_2\rightarrow \eta\times X_1$
given by $f(\xi,\alpha)=(\xi,f_{X_1X_2}(\alpha))$ is an isomorphism between the orders,
in particular
$$\ll_{1}\cap [\eta\times (X_1\cap X_2)]^2=\ll_{2}\cap [\eta\times (X_1\cap X_2)]^2.$$
Then there is $\eta<\eta'<\kappa$ and an $(\eta'\times X, \ll)$-order which is admissible such that
$(\eta\times X_1, \ll_1)$, $(\eta\times X_2,\ll_2)$  are good suborders of $(\eta'\times X, \ll)$.
\end{lemma}

\begin{proof}

First define an $\eta\times X$-order $\ll^*$ on $X=X_1*X_2$ which is not admissible but
such that $(\eta\times X_1, \ll_1)$, $(\eta\times X_2,\ll_2)$  are good suborders
 of $(\eta\times X, \ll^*)$.  Put
$s\ll^*t$ if and only if
 $s\ll_1t$ for $s,t\in\eta\times X_1$ or $s\ll_2t$ 
for $s,t\in\eta\times X_2$ or $s\ll_1 f(t)$ for $s\in\eta\times (X_1\setminus X_2)$
$t\in\eta\times (X_2\setminus X_1)$.

One proves that $\ll^*$ is a partial order indeed. For
this we note that $u\ll_1 s\ll_1 f(t)$ implies $u\ll_1 f(t)$
and that $s\ll_1 f(t)$, $t\ll_2 u$ implies $s\ll_1 f(u)$
as $f(t)\ll_1f(u)$ since  $f$ is an isomorphism.
$\ll^*$ clearly extends the orders $\ll_1$, $\ll_2$.

Then we note that the infima from $\ll_1$ and $\ll_2$
are preserved. One needs to check just $s,t\in X_2$.
Take $u\ll^*s,t$, one may assume that $u\in \eta\times (X_1\setminus X_2)$,
so $u\ll_1, f(t), f(s)$, so $f^{-1}(u)\ll_2 s, t$, this gives
$u\ll^* f^{-1}(u)\ll^* i_{X_2}(s,t)\ll^* s,t$ as required.

Finally let us prove the existence of the infimum for
$s\in \eta\times(X_1\setminus X_2)$ and $t\in \eta\times(X_2\setminus X_1)$.
Note that in that case
$\{u: u\ll^* s, t\}=\{u: u\ll_1 s, f(t)\}$ which has the biggest element $i_{X_1}(s, f(t))$,
thus (2) of Definition \ref{order} is satisfied.

Now it is enough to find $\eta<\eta'<\kappa$ and
an $(\eta'\times X)$-order $\ll$ which is admissible  and such that
$(\eta\times X, \ll^*)$ is a good suborder of $(\eta'\times X, \ll)$.
However, we will consider one more intermediate step.

Let $((\eta\xi, \eta(\xi+1)]\times X_1, \ll^\xi_1)$ for $0<\xi<\eta$ be  copies
of $(\eta\times X_1, \ll_1)$. Let $\alpha_1=\min(X_2\setminus X_1)$. Define
an $(\eta^2\times X_1)\cup[\eta\times (X_2\setminus X_1)]$-order $\ll^{**}$
by
\begin{itemize}
\item declaring $(\theta,\alpha)$ and $(\theta',\beta)$ incomparable
if $\alpha,\beta\in X_1$ and $\theta\in (\eta\xi, \eta(\xi+1)]$,
$\theta'\in (\eta\xi', \eta(\xi'+1)]$ for distinct $0\leq \xi,\xi'<\eta$,
\item sticking $(\eta(\xi+1), \eta(\xi+2)]\times X_1$ below 
$(\xi,\alpha_1)$ for each $0\leq \xi<\omega$.
\item sticking $(\eta\xi, \eta(\xi+1)]\times X_1$ below 
$(\xi,\alpha_1)$ for each $\omega\leq \xi<\eta$.
\end{itemize}
Of course after ``sticking" we make sure the new order is transitive by taking the transitive closure.
In fact, we just want to impose the admissibility condition which will fail for
$(\eta\times X, \ll^*)$ at elements $(\xi,\alpha)$ for $\xi<\eta$ and $\alpha\in X_2\setminus X_1$,
so  below these elements we stick some elements of the form $(\zeta, \beta)$ for 
$\eta<\zeta<\eta^2$ and $\beta\in X_1$.
We used copies of $(\eta\times X_1, \ll_1)$ because they are at hand (and are admissible), but 
most other choices would work if we do it in the incomparable manner as above.
We leave checking  the details of the fact that $\ll^{**}$ is  an $(\eta^2\times X_1)\cup[\eta\times (X_2\setminus X_1)]$-order
such that $\ll^*$ is a good suborder of it to the reader: the only nontrivial case for
checking the preservation of the $\ll^*$-suprema $i(s,t)$ is for
$s=(\xi,\alpha), t=(\xi', \beta)$ where $\xi, \xi'<\eta$ and $\alpha, \beta\in X_2\setminus X_1$;
but new elements $u=(\xi'', \gamma)$ below $s$ and $t$ must be for
$\gamma\in X_1$ and there must be a unique $w=(\xi''', \alpha_1)$ satisfying
$u\ll^{**}w\ll^{*}s, t$, so  $u\ll^{**}w\ll^{**}i(s,t)\ll^{**}s,t$
as required for the preservation of the suprema.

$\ll^{**}$ is a good extension of $\ll^{*}$ and so of $\ll_1$ and $\ll_2$ but
its domain is not of the form $\eta'\times X$ for $\eta'<\kappa$ and $X\in\mu$.
The last modification of $\ll^{**}$ aims at correcting this deficiency.
Using the fact that 
$\eta^2\eta^\omega=\eta^{2+\omega}=\eta^{1+\omega}=\eta\eta^\omega=\eta^\omega$
(with the ordinal exponentiation) we can construct  disjoint, incomparable  $\eta^\omega$-many  consecutive copies of 
$((\eta^2\times X_1)\cup[\eta\times (X_2\setminus X_1)], \ll^{**})$ 
with domains $(([\eta^2\xi, \eta^2\xi+\eta^2)\times X_1)\cup[\eta\xi, \eta\xi+\eta)\times (X_2\setminus X_1)]$
for $\xi<\eta^\omega$ and take their incomparable
union $\ll$ which will be a $(\eta^\omega\times X)$-order which is admissible and
such that $((\eta^2\times X_1)\cup[\eta\times (X_2\setminus X_1)], \ll^{**})$
 is a good suborder of it which completes the construction.

\end{proof}

\noindent{\bf Proof of Theorem \ref{thintall}}
\begin{proof}
By recursion on $X\in \mu$ we construct  an ordinal $\eta_X$ and an $\eta_X\times X$-order
$(\eta_X\times X, \ll_X)$ which is admissible, so that if $X, Y$ are elements of $\mu$ with $X\subseteq Y$, then
$\eta_X\leq \eta_Y$ and
$(\eta_X\times X, \ll_X)$  is a good suborder of $(\eta_Y\times Y, \ll_X)$.
 The lemmas \ref{limitorder} and \ref{successororder}
allow us to make the recursive step in such a way (i.e., looking backwards in $\kappa^+$) that 
 $(\eta_X\times X\cap\alpha, \ll_X)$ agrees with $(\eta_Y\times Y\cap\alpha, \ll_Y)$
whenever $X, Y\in\mu$ are of the same rank and $\alpha\in X\cap Y$, hence
$(\eta_X\times X\cap Y, \ll_X)$ agrees with $(\eta_Y\times X\cap Y, \ll_Y)$ and so
 the hypothesis
needed for Lemma \ref{successororder} is always present. The final $\kappa\times\kappa^+$-order
is obtained by applying Lemma \ref{limitorder} to $\nu=\mu$.

\end{proof}

\vskip 26pt

\section{Gaps, nonreflection and incompactness}\par
\vskip 13pt
 A natural phenomenon which accompanies constructions 
along $2$-cardinals are gaps, i.e.,
the fact that for given two cardinal invariants $\phi,\psi$
(i.e, some general way of assigning cardinal numbers
to structures of the type in question e.g. 
the width and the number of branches of trees or
the tightness and the character of points in
compact spaces) there is a  cardinal $\kappa$ such that
$\phi({\mathcal A})<\kappa<\psi({\mathcal A})$, where ${\mathcal A}$ 
is the constructed object.\par
A natural ``scenario" for  constructions of
objects exemplifying gaps goes as follows. In the 
inductive step we preserve enough properties or auxiliary 
objects so that
the invariant $\phi$ stays below $\kappa$, by 
preservation argument. On the other hand the inductive step
guarantees that properties or auxiliary objects 
involved in the definition of the invariant $\psi$ are 
not preserved. The number of constructions as 
above in combinatorial set theory is very large. Applications beyond set theory include
a construction of a large L-space (see \cite{hajnaljuhasz2}), a 
construction of a large Lindel\"of space with points $G_\delta$
see \cite{isaac}, \cite{vellemandiam}  (originally in  \cite{shelah1}), 
a Banach space of density $\omega_2$ without uncountable biorthogonal systems (\cite{christina}), 
etc.

As an example of an object
exhibiting  a gap, let us construct a Kurepa tree 
using a $(\kappa,\kappa^+)$-cardinal. The constructed Kurepa tree
has many additional properties, to mention only,
a nice well-ordering of branches. Note that P. Komjath 
has shown that a Kurepa tree with many properties of the
tree below may exist in a model where there is 
no morass and where  even $\square_{\omega_1}$ fails
(under some large cardinal assumption, see
\cite{komjath}). \par 
\vskip 13pt 
\begin{theorem}\label{kurepa}\cite{devlin2} Let $\kappa$ be a regular cardinal. If there is  a 
$(\kappa,\kappa^+)$-cardinal 
 then there is a 
  $\kappa$-Kurepa tree.
\end{theorem}
\begin{proof} Let $\mu$ be a $(\kappa,\kappa^+)$-cardinal.
Define $F\subset\kappa^{\kappa}$ as follows: 
$F=\{f_\alpha:\ \alpha<
\kappa^+\}$ where 
$$ f_\alpha(\xi)=ordtp(\mu_\xi(\alpha)),$$
where $ (\mu_\xi(\alpha))_{\xi<\kappa}$ is the $\mu$-sequence at $\alpha$.\par
First note that all $f_\alpha$'s are different. 
Let $\alpha\not=\beta$,
take $X\in{\mu}$ such that $\alpha,\beta\in X$ 
(there exists such an $X$
since ${\mu}$ is directed and 
covers $\kappa^+$, see definition 1.1. (6)), then
$ordtp(X\cap\alpha)\not=ordtp( X\cap\beta)$, 
this implies
that $f_\alpha(rank(X))\not=f_\beta(rank(X))$.\par
Now prove that for every $\xi<\kappa $ the cardinality of
$F\restriction\xi=\{f_\alpha\restriction\xi:\ 
\alpha<\kappa^+\}$ is less than 
$\kappa$. Take $X\in{\mu}$ such that
$rank(X)=\xi$. We will show that $F\restriction\xi
\subseteq\{f_\alpha
\restriction\xi:\
\alpha\in X\}$. 
This will suffice since $|X|<\kappa$. \par
Let us take arbitrary $\beta\in\kappa^+$, we can  find $Y\in{\mu}$
such that $rank(Y)=\xi$ and $\beta\in Y$. 
Since ${\mathcal \mu}\restriction X$ and
${\mathcal \mu}\restriction Y$ are isomorphic 
by homogeneity of ${\mu}$, there is $\alpha\in X$ such that
$f_{XY}(\beta)=\alpha$, then $f_\alpha\restriction\xi=
f_\beta\restriction\xi$ by the homogeneity \ref{morass}.\par
Hence  
$\{f_\alpha\restriction\xi:\ \alpha<\kappa^+,\ 
\xi<\kappa\}$ with the end-extension of functions is a subtree of $\kappa^{<\kappa}$ of height
$\kappa$  with levels of sizes $<\kappa$ 
with  at least $\kappa^+$-many branches of length $\kappa$, i.e.,  it is
a $\kappa$-Kurepa tree.
\end{proof}

\begin{theorem}\label{kurepastat}Suppose that $\kappa$ is a regular cardinal and that ${\mu}$ is a stationary 
$(\kappa,\kappa^+)$-cardinal, then there is a $\kappa$-Kurepa
tree with exactly $\kappa^+$-many branches of length $\kappa$ that does not 
contain a $\kappa$-Aronszajn subtree.
\end{theorem}
\begin{proof}   Our
$\kappa$-Kurepa 
tree $T$ with the above properties is the same as in the proof of
Theorem \ref{kurepa}, i.e., $T=\{f_\alpha\restriction\xi:\ \alpha<\kappa^+,\
\xi<\kappa\}$, so we adopt the notation of this proof and we 
will use the following observation made during the course of 
that proof
$$ F\restriction\xi=\{f_\alpha\restriction\xi:\ \alpha\in
X\}\leqno *)$$
for any $X\in{\mu}$ such that $rank(X)=\xi$.\par
a) Any branch of length $\kappa$ of $T$ is of the form $f_\alpha$
for some $\alpha<\kappa^+$.\par
\noindent Let $b$ be a branch of length $\kappa$ through $T$. Find an
elementary  submodel $M\prec
H(\kappa^{++})$ such that $M\cap \kappa^+\in {\mu}$
and  that $T,{\mu},b\in M$, $|M|<\kappa$. 
This is possible since ${\mu}$ is 
stationary in $\wp_\kappa(\kappa^+)$. 
Let $\xi=rank(M\cap\kappa)$. 
Let $f_\alpha$ be such
that $\alpha\in M\cap\kappa^+$ and 
$\{f_\alpha\restriction\eta: \eta<\xi\}=b\restriction\xi$
(by *)). Then since $\xi\not\in M$ by \ref{musequencemodel} (2), 
$M\models b=f_\alpha$, so
$H(\kappa^{++})\models b=f_\alpha$ so $b=f_\alpha$.\par
b) $T$ does not contain a $\kappa$-Aronszajn subtree.\par
\noindent Suppose $A$ is a subtree of $T$ of height $\kappa$. Let
$M\prec H(\kappa^{++})$ be a model  of cardinality less than $\kappa$
such that $T,{\mu},A\in M,\
M\cap\kappa^+\in{\mu}$. Let $t\in A$ be such that 
$ht(t)>rank(M\cap\kappa^+)=\xi$. By *) there is $\alpha\in M\cap\kappa^+$
such that
$f_\alpha\restriction\xi=t\restriction\xi$, so since $\xi\not\in M$,
$M\models (\{f_\alpha|\beta: \beta<\kappa\}\cap A\ is\ of\ size\ \kappa)$ so 
$\{f_\alpha|\beta: \beta<\kappa\}\cap A$ is
of size $\kappa$, hence $A$ has a $\kappa$-branch, so $A$ is 
not a $\kappa$-Aronszajn subtree.
\end{proof}

The fact that 
the statement of the theorem above   holds in
$L$ was originally proved in \cite{devlin1} and is
due to Jensen.  Note the 
inductive character of the above construction.
At the stage $X\in \mu$, we are given an initial
fragment of a Kurepa tree. The {\it coherence} of a morass
guarantees that, different interpretations of 
the set of branches of this fragment of the tree are
consistent. This is the case
when a recursive construction as in the previous section can be easily made explicit. We extend the 
tree at successor stages, splitting a branch
if it corresponds to an element from the {\it tail}
and leaving a branch non-split if it is in the {\it head}
of the $\Delta$-system given by amalgamation pair
at the considered rank.  Also, the way the gap between
the number of branches and the size of the levels is
obtained is evident: at the stage of successor rank,
we preserve the level but increase the set of branches.

\noindent Another natural phenomenon occurring while
 sophisticated stepping up principles are allowed to work  is 
the nonreflection, i.e., the nonexistence of a substructure
to which a given structure reflects its given properties, e.g., 
a nonmetrizable space all of whose small subspaces are metrizable.
The small size of initial fragments of the construction
is responsible for obtaining a given property in 
substructures of size less than $\kappa$. The pressing down
lemma applied to e.g. a stationary $2$-cardinal proves that 
the entire structure does not have a given property $P$. 
So, often the stationary nonreflection is the underlying one, hence
in this context it is natural to consider stationary stepping up tools.
 (see \cite{stevorado}, \cite{semi}, \cite{kur} or in the $\square_\kappa$ context
e.g. \cite{hajnaljuhasz1}).

\begin{proposition}[\cite{semi}]
 Let $\mu$ be a $(\kappa,\kappa^+)$-cardinal,
 then for no proper subset $A\subset\kappa^+$ of size 
at least $\kappa$ 
the set $\{X\in\mu:\  X\subset A\}$  is stationary in
$\wp_{\kappa}(A)$.
\end{proposition}
\begin{proof} We will build a regressive
function $f:\{X\in\mu:\ X\subseteq A\}\rightarrow
A$ such that for each $\alpha\in A$ there is a bound
in $\kappa$ for ranks of all elements of $\mu$ in 
the preimage of $f^{-1}(\{\alpha\})$. This will
be sufficient, since  for $\kappa$ regular,
no well-founded cofinal set
in $\wp_\kappa(A)$ can have bounded ranks 
(consider the union of representatives of each rank). 
Hence the function as above will have nonstationary
preimages of singletons, thus by the pressing down lemma
(see \cite{baumgartnerhandbook}) its domain is nonstationary.

First choose $\beta\in \kappa^+$ such that $\beta\not\in 
A$, then
$f(X)\in X$ is such that
$$ordtp(X\cap f(X))=ordtp(Y\cap\beta)$$
where $Y\in\mu$  is such that $\beta\in Y, rank(Y)=rank(X)$.
Note that $f$ is well-defined. This follows from the
density lemma \ref{density}  and the coherence lemma \ref{coherence}
and Lemma \ref{zerorank}.

Suppose that $\alpha\in A$. Then there is $Z\in \mu$ 
such that $\alpha,\beta\in Z$, we will prove that
$rank(Z)$ bounds ranks of elements 
in $f^{-1}(\{\alpha\})$. Let $f(X)=\alpha$, $rank(X)>rank(Z)$,
so by the density lemma \ref{density}
there is $Y\in \mu$ of the same rank as the rank of $X$ such that $\alpha,\beta\in Y$,
then obviously $ordtp(\alpha\cap Y)\not=ordtp(\beta\cap Y)$
and $X\cap\alpha=Y\cap \alpha$ by the coherence lemma \ref{coherence}, so
$ordtp(\alpha\cap X)\not=ordtp(\beta\cap Y)$, but
this contradicts the fact that $f(X)=\alpha$.  \end{proof}

\begin{corollary}
Let $\mu$ be a $(\kappa,\kappa^{+})$-cardinal,
then  $\mu$ is not a $club$ subset of $\wp_\kappa(\kappa^+)$.
\end{corollary}

The last example in this section concerns the Hausdorff gap and its generalizations to
higher cardinals whose consistency is originally proved in \cite{negrepontis}.
Hausdorff gaps can be considered as objects exhibiting nonreflection. The entire
two chains of regular length $\kappa$ cannot be separated, but this property does not reflect to chains of
smaller sizes (included in initial chains) which
can be separated.
Below in the case of $\kappa=\omega$ we obtain an explicit definition of the classical Hausdorff gap
in ZFC because $(\omega,\omega_1)$-cardinals exist in ZFC as proved in \cite{vellemanzfc}.

\begin{theorem}[\cite{vellemanzfc}] Suppose that $\kappa$ is a regular cardinal and that
there exists  a $(\kappa,\kappa^+)$-cardinal $\mu\subseteq \wp_\kappa(\kappa^+)$.
Then there are $(A_\alpha)_{\alpha<\kappa^+}, (B_\alpha)_{\alpha<\kappa^+}\subseteq \wp(\kappa)$
such that
\begin{enumerate}
\item $A_\alpha\cap B_\alpha=\emptyset$ for each $\alpha<\kappa^+$,
\item $|A_\alpha\setminus A_\beta|, |B_\alpha\setminus B_\beta|<\kappa$ for each $\alpha<\beta<\kappa^+$,
\item There is no $C\subseteq\kappa$ such that $|A_\alpha\setminus C|, |B_\alpha\cap C|<\kappa$
for each $\alpha<\kappa^+$.
\end{enumerate}
\end{theorem}
\begin{proof} 
Define 
$$A_\alpha=\{\xi\in \kappa:  \exists X_1, X_2\in \mu \
rank(X_1)=rank(X_2)=\xi, X_1*X_2\in \mu\ \hbox{and} \  \alpha\in X_1\setminus X_2\}$$
$$B_\alpha=\{\xi\in \kappa:  \exists X_1, X_2\in \mu \
rank(X_1)=rank(X_2)=\xi, X_1*X_2\in \mu\ \hbox{and} \  \alpha\in X_2\setminus X_1\}$$

To prove (1), suppose that $\xi\in X_1, X_2'$ and there are
$X_2$ and $X_1'$ such that $X_1*X_2, X_1'*X_2'\in\mu$ are of rank $\xi+1$. This 
contradicts the coherence lemma \ref{coherence} and the homogeneity \ref{morass} (3). 
To prove (2) suppose that $\alpha<\beta<\kappa^+$ and
that $\eta\in\kappa$ is above the rank of some $Y\in \mu$ satysfying
$\alpha,\beta\in Y$.
We will note that whenever $X\in \mu$ and $rank(X)=\xi+1\geq \eta$,
$X=X_1*X_2$ and $\alpha\in X_i\setminus X_{3-i}$, then
$\beta\in X_i\setminus X_{3-i}$ as well. This follows from the fact that
there is $Y'\in \mu| X_i$ for some $i=1, 2$ such that $\alpha, \beta\in Y'$
which is a consequence of the local almost directedness \ref{morass} (5b) and the
localization lemma \ref{localization}.

Finally let us see why (3) holds.  For $\alpha<\kappa^+$ let $f_\alpha:\alpha\rightarrow\kappa$
be a function defined for $\beta<\alpha$ by
$$f_\alpha(\beta)=\min\{\xi: (A_\alpha\cap B_\beta), (B_\alpha\cap A_\beta)\subseteq \xi\}.$$
(1) and (2) imply that $f$ is well-defined. As in the case of the classical Hausdorff gap construction it will be enough to
prove that the preimages of singletons under $f_\alpha$'s have cardinalities less than $\kappa$.
We will denote this statement as (*).
Indeed, under this hypothesis, if there were $C\subseteq \kappa$ as in
(3), then  for $\kappa^+$-many $\beta\in \kappa^+$ there would exist $\xi<\kappa$ such that
$$A_\beta\setminus C, B_\beta\cap C\subseteq \xi$$
Take an $\alpha_0$ among these $\{\beta: A_\beta\setminus C, B_\beta\cap C\subseteq \xi\}=D_\xi$
 such that below $\alpha_0$ there are $\kappa$ many elements of $D_\xi$.
Then $f_{\alpha_0}\restriction D_{\xi}$ assumes all its values below $\xi<\kappa$ and so
one value is assumed on $\kappa$ many elements by the regularity of $\kappa$,
contradicting the statement (*) about the $f_\alpha$s.

To prove  (*) fix $\xi<\kappa$ and $\alpha\in\kappa^+$. Let
$X\in \mu$ be of rank $\xi$ such that $\alpha\in X$. We will show that
for each $\beta\in \alpha\setminus X$ we have $f_\alpha(\beta)>\xi$ which is 
enough for (*).  Take $Y\in \mu$ of minimal rank such that
$\beta,\alpha\in Y$. By the coherence
Lemma \ref{coherence} $rank(Y)>\xi$.

 By the density lemma \ref{density}
and the coherence lemma \ref{coherence} we may assume that $X\subseteq Y$.
It follows from the local almost directedness \ref{morass} (5) that
$Y=Y_1*Y_2$. By the minimality of the rank of $Y$ we have that
$\beta\in Y_1\setminus Y_2$ and $\alpha\in Y_2\setminus Y_1$ and hence
$rank(Y)-1\in A_\beta\cap B_\alpha$ and so $f_\alpha(\beta)> rank(Y)-1\geq\xi$
as required.

\end{proof}

\vskip 26pt
\section{Coherent partitions of pairs}
\vskip 13pt
In this section we show a way of working with $2$-cardinals 
parallel to the methods of walks on ordinals introduced and developed by S. Todorcevic
(for a survey see \cite{stevowalks}). 
Todorcevic proved in ZFC  (\cite{stevoacta}) a strong failure of the Ramsey property
at $\omega_1$ and developed methods of stepping up (this failure and other phenomena) 
to higher cardinals based on the assumption of $\square_{\kappa}$ and using
colorings $\rho:[\kappa^+]^2\rightarrow \kappa$  with some stronge coherence properties
(see \cite{stevoacta} Section 2).
It was C. Morgan (Definition 2 of \cite{morganforcing}) who realized that using a 
simplified morass one can define colorings sharing many properties with $\rho$\footnote{For $\kappa=\omega_1$ 
the existence of $(\omega_1,\omega_2)$-cardinal implies $\square_{\omega_1}$, but 
it does not hold for other $\kappa$'s (\cite{vellemandiam}).
The opposite implication  does not hold even for $\kappa=\omega_1$
as the consistency strength of the negation  of $\square_{\omega_1}$
is the existence of a Mahlo cardinal  (see \cite{devlin2}) and the 
consistency strength of the nonexistence of an 
$(\omega_1,1)$-morass is the existence of an 
inaccessible cardinal (see \cite{devlin2}).}.

In this section 
after the proof of the fundamental properties of the 
colorings we use them for the proof of the existence of  $\kappa^+$-Aronszajn tree
and  the existence of a function with  property $\Delta$.
Our arguments work in a new context
 of $\kappa>\omega_1$ apparently not addressed in the literature before.
This presentation is very modest compared to the applications of 
$\rho$-functions which resulted in the case of $\kappa=\omega$ in
many fascinating constructions (see \cite{stevowalks}, \cite{stevocoherent})
for example of Banach spaces (e.g. \cite{argyrosjordi}) extraspecial p-groups (\cite{sspgrups}),
quadratic vector spaces (\cite{quadratic}), zero-sets of polynomials in the infinite dimension
(\cite{antonio}) and many others.
The main results concerning coherent partitions of pairs which 
are present in the literature at the moment in a language which can be easily interpreted
in the context of $2$-cardinals concern generic stepping up and are addressed in the next section.
Also the main applications of  property $\Delta$ discussed in this section belong there.

\vskip 13pt
\begin{definition}[\cite{morganforcing}]\label{todorcevic}
 Let $\mu$ be a $(\kappa,\kappa^+)$-cardinal,
then the following function 
$m_\mu=m\colon[\kappa^+]^2\rightarrow\kappa$
is called a $\mu$-coloring:
$$m(\alpha,\beta)=m(\{\alpha,\beta\})=\min\{rank(X):\ \alpha,\beta\in X\in\mu\}$$
\end{definition}

The coherence of $\mu$-sequences translates into the coherence of $\mu$-colorings.

\begin{lemma}\label{mucoloringcoherence}
Suppose that $\kappa$ is a regular cardinal, $\beta<\alpha<\kappa^+$,  $\mu$ 
is a $(\kappa,\kappa^+)$-cardinal and $(\mu_\xi(\alpha))_{\alpha<\kappa}$ is a $\mu$-sequence at $\alpha$
as defined in \ref{musequencedef}. Then
for every $\xi\geq m(\alpha,\beta)$ we have 
$$\mu_{\xi}(\beta)=\mu_\xi(\alpha)\cap\beta.$$
\end{lemma}
\begin{proof}This follows from \ref{musequencecoherence} and the definition of $\mu$-coloring.
\end{proof}

The following proposition corresponds to 2.3 of \cite{stevoacta}. 

\vskip 13pt
\begin{proposition}\label{rho} Let $\kappa$ be a regular cardinal and
$\mu$ be a $(\kappa, \kappa^+)$-cardinal. Let $m\colon[\kappa^+]^2
\rightarrow\kappa$ be a $\mu$-coloring.
Let $\alpha<\beta<\gamma<\kappa^+$, $\nu<\kappa$,
$0<\delta=\bigcup\delta<\epsilon<\kappa^+$, then
the following conditions are satisfied:
\begin{enumerate}[(a)]
\item $|\{\xi<\alpha:\ m(\xi,\alpha)\leq\nu\}|<\kappa$

\item $m(\alpha,\gamma)\leq\max\{m(\alpha,\beta),m(\beta,\gamma)\}$
\item $m(\alpha,\beta)\leq\max\{m(\alpha,\gamma),m(\beta,\gamma)\}$

 \item There is  
 $\zeta<\delta$ such that $m(\xi,\epsilon)\geq m(\xi,\delta)$ for all $\zeta\leq\xi<\delta$.
\end{enumerate}

\end{proposition}

\begin{proof}
\noindent (a)\par
Let $(\mu_\xi)_{\xi<\kappa}(\alpha)$ be a $\mu$-sequence at $\alpha$ (see \ref{musequencedef}).
By the definition of $m$ and the coherence lemma \ref{coherence} and Lemma
\ref{musequencelemma} the following  is satisfied:
$$\{\xi<\alpha:\ m(\xi,\alpha)\leq\nu\}= \mu_\nu(\alpha).$$

\noindent (b)\par
Let $X,Y\in\mu$ be such that
$$\alpha,\beta\in X,\ rank(X)=\max\{m(\alpha,\beta),
m(\beta,\gamma)\}$$
$$\beta,\gamma\in Y,\ rank(Y)=\max\{m(\alpha,\beta),
m(\beta,\gamma)\},$$
which exist by the definition of $m$ and the density lemma \ref{density}.
Now $\beta\in X,Y\in\mu,\ rank(X)=rank(Y)$, so
$X\cap\beta=Y\cap\beta$ by the coherence  lemma \ref{coherence},
so $\alpha\in Y$, and hence $m(\alpha,\gamma)
\leq rank(Y)=\max\{m(\alpha,\gamma),m(\beta,\gamma)\}$.\par
\vskip 6pt
\vfill
\break
\noindent (c)\par
Let $X,Y\in\mu$ be such that
$$\alpha,\gamma\in X,\ rank(X)=\max\{m(\alpha,\gamma),
m(\beta,\gamma)\},$$
$$\beta,\gamma\in Y,\ rank(Y)=\max\{m(\alpha,\gamma),
m(\beta,\gamma)\}.$$
As $\gamma\in X,Y\in\mu,\ rank(X)=rank(Y)$, so
$X\cap\gamma=Y\cap\gamma$ by the coherence lemma \ref{coherence},
so $\alpha\in Y$, and hence $m(\alpha,\gamma)
\leq rank(Y)=\max\{m(\alpha,\gamma),m(\beta,\gamma)\}$.\par
\vskip 6pt
\noindent(d) 
We will prove it  by induction on $m(\delta,\epsilon)$.
Let 
$X\in\mu$ be  of minimal rank which contains $\delta$ and $\epsilon$.
Note that if
$\xi<\delta$ and $\xi\not \in X$, then by \ref{musequencelemma} 
for
$\mu$-sequence at $\epsilon$ any element of
$\mu$ which contains $\xi$ and $\epsilon$ must contain $\delta$, and so
$m(\xi,\epsilon)\geq m(\xi,\delta)$, as required.

Now let us turn to $\xi\in X$. By the neatness $X=X_1*X_2$. By the minimality $\delta\in X_1\setminus X_2$
and $\epsilon\in X_2\setminus X_1$. 

If $\zeta=\sup(X_1\cap X_2)<\delta$,
then note that any $\xi\in X$ satisfying $\zeta\leq \xi<\delta$ belongs to
$X_1\setminus X_2$ and so $m(\xi,\epsilon)=rank(X)>rank(X_1)\geq m(\xi,\delta)$.

If $\sup(X_1\cap X_2)=\delta$, we consider two cases. First $f_{X_1X_2}(\epsilon)=\delta$,
then $m(\xi,\epsilon)=m(\xi,\delta)$ for all $\xi\in X_1\cap X_2=\{\xi\in X: \xi<\delta\}$
by the homogeneity of $\mu$. Secondly $\delta< f_{X_1X_2}(\epsilon)$, then
we use the inductive assumption to conclude that there is
$\zeta<\delta$ such that $m(\xi, f_{X_1X_2}(\epsilon))\geq m(\xi,\delta)$ for
every $\zeta\leq\xi<\delta$. However $m(\xi, f_{X_1X_2}(\epsilon))=m(\xi,\epsilon)$
for $\xi \in X_1\cap X_2$ by the homogeneity of $\mu$ and in this case
$X_1\cap X_2=\{\xi\in X: \xi<\delta\}$. This completes the proof of (d).
\end{proof}

\begin{corollary}\label{mequivalence}
Suppose that $\kappa$ is a regular cardinal, $\mu$ is a $(\kappa,\kappa^+)$-cardinal
and $m$ is the $\mu$-coloring.
Let $\gamma_1<\gamma_2<\gamma_3<\kappa^+$. Then
\begin{enumerate}
\item $m(\gamma_1,\gamma_2)\leq m(\gamma_2,\gamma_3)\ \ \hbox{ if and only if}\ \
m(\gamma_1,\gamma_3)\leq m(\gamma_2,\gamma_3)$;
\item if $m(\gamma_1, \gamma_2)>m(\gamma_2, \gamma_3)$
or $m(\gamma_1, \gamma_3)>m(\gamma_2, \gamma_3)$, then $m(\gamma_1,\gamma_2)=
m(\gamma_1, \gamma_3)$.
\end{enumerate}
\end{corollary} 
\begin{proof} (b) and (c) of \ref{rho} assume the following forms
 $$m(\gamma_1,\gamma_3)\leq \max\{m(\gamma_1, \gamma_2), m(\gamma_2,\gamma_3)\}.\leqno (*)$$
$$m(\gamma_1,\gamma_2)\leq \max\{m(\gamma_1, \gamma_3), m(\gamma_2,\gamma_3)\}.\leqno (**)$$

 (1) For the forward implication, use the hypothesis and (*).
For the backward implication, use the hypothesis and (**).

(2) In the first case, the hypothesis $m(\gamma_1, \gamma_2)>m(\gamma_2, \gamma_3)$  and (**) gives
$m(\gamma_1,\gamma_2)\leq m(\gamma_1, \gamma_3)$ while the hypothesis and  (*)
gives $m(\gamma_1,\gamma_3)\leq m(\gamma_1, \gamma_2)$.
In the second case, the hypothesis $m(\gamma_1, \gamma_3)>m(\gamma_2, \gamma_3)$  and (*) gives
$m(\gamma_1,\gamma_3)\leq m(\gamma_1, \gamma_2)$ while the hypothesis and  (**)
gives $m(\gamma_1,\gamma_2)\leq m(\gamma_1, \gamma_3)$.
\end{proof}

\begin{theorem} [\cite{devlin2}]
Let $\kappa$ be a regular cardinal and $\mu$ a $(\kappa,\kappa^+)$-cardinal.
Suppose that $m:[\kappa^+]^2\rightarrow \kappa$ is 
a $\mu$-coloring.
 Then $T=\{m(^.,\alpha)\restriction\beta:\ \beta<\alpha<\kappa^+\}$
 with inclusion is a $\kappa^+$-Aronszajn tree.
\end{theorem}
\begin{proof} The proof follows \cite{stevoacta}.
First note that $T$ does not have branches of length $\kappa^+$.
Since each function $m(^.,\alpha)$ is $<\kappa$-to-one
(by \ref{rho} (a)), as $\kappa^+$ is regular, a branch of length $\kappa^+$ would 
give rise to $<\kappa$-to-one function from $\kappa^+$
into $\kappa$ which is impossible.\par
\noindent It can be easily seen that $Lev_\beta(T)=
\{m(^.,\alpha)\restriction\beta:\beta<\alpha<\kappa^+\}$.
We need to show that this set has size at most $\kappa$.
Let us define a relation for $\alpha_1,\alpha_2\in \kappa^+-\beta$
by
$\alpha_1=_\beta\alpha_2$  if and only if
$$\exists X_1,X_2\in\mu\ rank(X_1)=rank(X_2),\
\alpha_1,\beta\in X_1,\ \alpha_2,\beta\in X_2,\
f_{X_1X_2}(\alpha_1)=\alpha_2.$$
By the fact that $f_{X_3 X_2}\circ f_{X_2 X_1}=f_{X_3 X_1}$
for $X_1, X_2, X_3\in \mu$ of the same rank, the $=_\beta$
is an equivalence relation. Note that there are 
at most $\kappa$-many equivalence
classes of this relation, as there are
$\kappa$-many ranks and each element of $\mu$ has less than 
$\kappa$ elements. So, it is sufficient to prove
that if $\alpha_1=_\beta\alpha_2$, then $m(^.,\alpha_1)\restriction
\beta=m(^.,\alpha_2)\restriction\beta$. 
Let $X_1, X_2$ witness the fact that
$\alpha_1=_\beta\alpha_2$. Let $\gamma<\beta$. 

If 
$\gamma\in X_1\cap\beta=X_2\cap \beta$, then 
$m(\gamma,\alpha_1)=m(\gamma,\alpha_2)$, since 
$f_{X_1 X_2}(\alpha_1)=\alpha_2$, and by the homogeneity 
of $\mu$.

 If $\gamma\not\in X_1\cap\beta=X_2\cap\beta$,
then $m(\gamma,\alpha_1)=m(\gamma,\alpha_2)$ by \ref{mequivalence} (2).
\end{proof}

 The existence of a $\kappa^{++}$-Souslin tree
may also follow from the existence of a $(\kappa,\kappa^+)$-cardinal.
It is so when $2^\kappa=\kappa^+$ or when a Cohen subset
of $\kappa^+$ is added generically to the universe
(see \cite{vellemandiam} or \cite{shelahstanley}).

\begin{proposition}\label{neatsuccessor}
Suppose that $\kappa$ is a regular cardinal and $\mu$ is a  $(\kappa,\kappa^+)$-cardinal.
All the values of $m_\mu$ are successor ordinals.
\end{proposition}
\begin{proof} This follows from the neatness of $\mu$ as in \ref{morass}.
\end{proof}

\begin{lemma}\label{mmodel} Let $\kappa$ be a regular cardinal and $\mu$ be a 
 $(\kappa,\kappa^+)$-cardinal. Suppose that $M\prec H(\kappa^{++})$ is an elementary
submodel which contains $\mu$. Let $\delta=M\cap\kappa\in\kappa$ and $\gamma_1<\gamma_2<\kappa^+$,
then
\begin{enumerate}
\item If $\gamma_1,\gamma_2\in M$, then $m(\gamma_1, \gamma_2)<\delta$ 
\item If $\gamma_1\not\in M$ and $\gamma_2\in M$, then $m(\gamma_1, \gamma_2)>\delta$
\end{enumerate}
\end{lemma}
\begin{proof}
(1) is clear as $m(\gamma_1, \gamma_2)$ is an element of $\kappa$
definable in $M$.  For (2) suppose
 that $m(\gamma_1, \gamma_2)>\delta$ does not hold and note that by \ref{neatsuccessor}
this means that $m(\gamma_1, \gamma_2)<\delta$, so  $m(\gamma_1, \gamma_2)$ is
in $M$ and hence
 $\mu_{m(\gamma_1, \gamma_2)}(\gamma_2)$ ($\mu$-sequence as in \ref{musequencedef})
is in $M$. Then it must be a subset of $M$ since $M\cap\kappa$ is an ordinal.
But $\gamma_1$ belongs to it, so $\gamma_1\in M$.

\end{proof}

In the case of a $\mu$-coloring where $\mu$ is a 2-cardinal we can obtain some more
concrete information corresponding to (a) and (d) of \ref{rho} included in the following two propositions.

\begin{proposition}Let $\kappa$ be a regular cardinal and
$\mu$ be a $(\kappa, \kappa^+)$-cardinal such that $|X|<rank(X)^+$ for all $X\in\mu$. 
Let $m\colon[\kappa^+]^2
\rightarrow\kappa$ be a $\mu$-coloring.
Let $\alpha<\beta<\gamma<\kappa^+$, $\nu<\kappa$. Then
$$|\{\xi<\alpha:\ m(\xi,\alpha)\leq\nu\}|<\nu^+ .$$

\end{proposition}
\begin{proof}
It is like (a) of \ref{rho}.
\end{proof}

\begin{proposition}\label{version4}
Suppose that $\kappa$ is a regular cardinal, $\mu$ is a  $(\kappa,\kappa^+)$-cardinal
and $m$ is the $\mu$-coloring. Let $\delta<\kappa$ be a limit ordinal and let $\tau<\epsilon<\kappa^+$.

\noindent There is $\zeta=\zeta(\tau,\epsilon,\delta)<\delta$  such that whenever $\xi<\tau$ satisfies
$\zeta(\tau,\epsilon,\delta)<m(\xi,\tau)<\delta$,
then $m(\xi,\tau)\leq m(\xi,\epsilon)$.
\end{proposition}
\begin{proof}
Let $X, Y\in\mu$ be such that $\tau\in X$, $\epsilon\in Y$,  and $rank(X)=rank(Y)=\delta$.
The existence of these sets follows from Lemma \ref{zerorank} and the density lemma.
Note that $X\cap Y<X\setminus Y, Y\setminus X$.
Using the homogeneity as in \ref{morass} there is an order preserving 
function $f_{XY}:Y\rightarrow X$ (If $X=Y$ we just put $f_{XY}=Id_X$).  

We claim that if $m(\tau,\epsilon)<\delta$, then $\zeta=m(\tau,\epsilon)$ works;
if $m(\tau,\epsilon)>\delta$ and 
$\tau\not=f_{XY}(\epsilon)$, then $\zeta=m(\tau, f_{XY}(\epsilon))+1$ works; 
and otherwise $\zeta=0$ works.
First note that $\zeta<\delta$ by \ref{neatsuccessor} in all these cases.
We will consider two cases with subcases.

\noindent{\bf Case 1.} $\tau\in X\cap Y$.

By \ref{neatsuccessor}, $\zeta=m(\tau,\epsilon)<\delta$.
Now if $\xi<\tau<\epsilon$ and $m(\xi,\tau)>\zeta$, we can apply \ref{mequivalence} (2)
to conclude that $m(\xi,\tau)=m(\xi,\epsilon)$, that is $\zeta=m(\tau,\epsilon)$ 
works.

\noindent{\bf Case 2.} $\tau\in X\setminus Y$.

The condition $m(\xi,\tau)<\delta$ from the statement of the proposition
yields $\xi\in X$, and so we may consider only $\xi \in X$.
Moreover, in  this case we may consider only
$\xi\in X\cap Y$, as the other $\xi$'s satisfying
$m(\xi,\tau)<\delta$, $\xi<\tau$ are in $X\setminus Y$
and so, since they satisfy $\xi<\epsilon$ (as $\tau<\epsilon$), we have that
$m(\xi,\tau)\leq\delta<m(\xi,\epsilon)$. So 
 $\zeta=0$ works for $\xi\in(X\setminus Y)\cap\tau$.

\noindent{\bf Case 2.1.} $f_{XY}(\epsilon)=\tau$.

Then   $m(\xi, \tau)=m(\xi, f_{XY}(\epsilon))=m(\xi, \epsilon)$
for $\xi\in X\cap Y$ by the homogeneity \ref{morass}.

\noindent{\bf Case 2.2.} $f_{XY}(\epsilon)\not =\tau$.

As we are in Case 2. we have $\tau\in X\setminus Y$ and so
$\epsilon\in Y\setminus X$ and so $f_{XY}(\epsilon)\in X\setminus Y$.
Since as before  we may assume that $\xi\in X\cap Y$, we
conclude that $\xi<\tau, f_{XY}(\epsilon)$.
In this situation,  if 
$$m(\tau, f_{XY}(\epsilon))<m(\tau, f_{XY}(\epsilon))+1=\zeta< m(\xi,\tau),$$ we may use
\ref{mequivalence} (2) to conclude that $m(\xi, \tau)=m(\xi, f_{XY}(\epsilon))$.
However as $\xi\in X\cap Y$ we have $m(\xi, f_{XY}(\epsilon))=m(\xi,\epsilon)$
which completes the proof.
\end{proof}

\begin{proposition}\label{velickovic} Let $\lambda$ be  an infinite regular cardinal,  
such that $\lambda^{<\lambda}=\lambda$ and  let $\kappa=\lambda^+$. Assume that 
$\mu$ is a $(\kappa,\kappa^+)$-cardinal and $m$ is the $\mu$-coloring. Suppose that
 $\{a_\xi:\ \xi\in\kappa\}$ is a collection of  subsets of $\kappa^+$ of cardinalities smaller than $\lambda$.
Then there is $A\subseteq\kappa$ of cardinality $\kappa$
such that for any  $\xi,\eta\in A$ we have satisfied the following relations:
if  $\tau\in a_\xi\cap a_{\eta}$, $\alpha\in a_\xi-a_{\eta}$,
 $\beta\in a_{\eta}-a_\xi$, then
\begin{enumerate}

\item $\beta>\tau\ \Rightarrow\ m(\alpha,\tau)\leq m(\alpha,\beta),$
\item $\alpha>\tau\ \Rightarrow\ m(\beta,\tau)\leq m(\alpha,\beta).$
\end{enumerate}
\end{proposition}
\begin{proof}

 Using the hypothesis $\lambda^{<\lambda}=\lambda$ we may apply the 
$\Delta$-system lemma (1.6. of \cite{kunen}) and we  may w.l.o.g. assume that 
$(a_\xi:\ \xi<\kappa)$ is a $\Delta$-system with  root $\Delta$.

If the proposition is false, there are $A_\theta\subseteq \kappa$ such that $|A_\theta|<\kappa$
and $A_\theta<A_{\theta'}$ for each $\theta<\theta'<\kappa$ such that 
for each $\theta<\kappa$ and for every  $A_\theta<\eta_\theta<\kappa$ 
there is $\xi\in A_\theta$ such that the  pair $\xi, \eta_\theta$  does not satisfy the relations as in
the proposition. Indeed, otherwise for some $\xi<\kappa$ one could build $A\subseteq \kappa\setminus\xi$ 
as in the proposition by recursion. So we will assume the existence of $A_\xi$s as above and will derive
a contradiction.

For the simplicity of the argument let us use an elementary submodel (see a
survey of A. Dow \cite{dow} for standard methods concerning the applications of
elementary submodels). So let
$M\prec H(\kappa^{++})$ be of cardinality
 $\lambda$ and such that $[M]^{<\lambda}\subseteq M$ and 
$\lambda,\mu, \{a_\xi: \xi<\kappa\}, \{A_\theta: \theta<\kappa\} \in M$.
Moreover let $\delta=M\cap\kappa\in\kappa$ be such  that $cf(\delta)=\lambda$.

Let $\eta<\kappa$ be such that $\eta\not\in M$. It follows that
$(a_\eta\setminus\Delta)\cap M=\emptyset$ as the elements of $a_\eta\setminus\Delta$ may belong to just one
set in $\{a_\xi\setminus\Delta:\xi<\kappa\}$, namely $a_\eta$.

Now we start the search for conditions on $\xi<\kappa$ which guarantee that
all the elements  $\alpha\in a_\xi\setminus a_\eta$,
$\beta\in a_\eta\setminus a_\xi$ and $\tau\in a_\xi\cap a_\eta$ satisfy (1) and (2). 
Later we will find a $\theta<\kappa$ with $A_\theta<\eta$ such that for each $\xi \in A_\theta$
the ordinals $\xi, \eta$ satisfy these conditions which will bring the required contradiction.

Let $\pi\in M$
be the minimal element of $\kappa^+$ bigger  than every element of  $a_\xi\setminus \Delta$
 for every $\xi<\kappa$.
Recalling \ref{version4} define:
$$\zeta=\sup\{\zeta(\tau, \beta,\delta):  \beta\in a_\eta\setminus\Delta, \tau \in \Delta\cap\beta \}.$$
As $cf(\delta)=\lambda$ and $\lambda$ is regular we conclude that
$$\zeta<\delta.\leqno (a)$$
Now let $\delta'<\kappa$ satisfy $\delta'<\delta$ and
$$\mu_{m(\tau,\beta)}(\beta)\cap M\subseteq \mu_{\delta'}(\pi)\leqno (b)$$
for any   $\beta\in a_\eta\setminus\Delta$ and any
$\tau\in\Delta$ with $m(\tau, \beta)<\delta$ .
The existence of such a  $\delta'$ for a single pair $\tau,\beta$ as above follows from
\ref{musequencecofcof} because $M\cap\beta\subseteq M\cap\pi=
\mu_\delta(\pi)$ by \ref{musequencemodel}. As $\Delta$ and $a_\eta\setminus\Delta$ have
cardinalites less than $\lambda$, the monotonicity of the $\mu$-sequence \ref{musequencelemma}
and $cf(\delta)=\lambda$ imply that we can find $\delta'$ that does the job
for all  $\tau$s and $\beta$s as above. 

\noindent{\bf Claim:} If $\xi\in M\cap\kappa$ satisfies

\begin{enumerate}[(c)]
\item[(c)] $(a_\xi\setminus \Delta)\cap\mu_{\delta'}(\pi)=\emptyset$,
\item[(d)]  $\zeta<m(\alpha,\tau)$ for every $\tau\in \Delta$ and $\alpha\in (a_\xi\setminus\Delta)\cap\tau$,
\end{enumerate}
then the relations from the statement of the proposition 
are satisfied for $\xi$ and $\eta$.

 \noindent Proof of the claim:
By \ref{mmodel} we can improve (d) to  
\begin{enumerate}[(d')]
\item $\zeta<m(\alpha,\tau)<\delta$ for every $\tau\in \Delta$ and $\alpha\in (a_\xi\setminus\Delta)\cap\tau$
\end{enumerate}
   as
$a_\xi,\tau\in M$. Let $\alpha, \beta, \tau$ be as in the proposition. Note
that we may assume that $\tau\not=\max\{\alpha,\beta,\tau\}$.

\noindent{\bf Case 1.} $\alpha=\max\{\alpha,\beta,\tau\}$.

We have $m(\tau,\alpha)<\delta<m(\beta, \alpha)$ by \ref{mmodel}.
As $m(\tau,\beta)\leq \max(m(\tau,\alpha), m(\beta,\alpha))$ by \ref{rho} (c)
we also have $m(\tau,\beta)\leq m(\beta,\alpha)$.

\noindent{\bf Case 2.} $\tau<\alpha<\beta$.

First assume that 
$m(\tau,\beta)<\delta$.  By (b) and (c) above  $m(\tau,\beta)<m(\alpha,\beta)$
since $\mu$-sequence at $\beta$ is  nondecreasing by \ref{musequencelemma}.
By \ref{rho} (c) we have $m(\tau, \alpha)\leq\max(m(\tau,\beta), m(\alpha,\beta))$
and so $m(\tau,\alpha)\leq m(\alpha,\beta)$ holds as well.

Now assume that 
$m(\tau,\beta)\geq\delta$ and so by \ref{mmodel} we 
have $m(\tau,\alpha)<\delta\leq m(\tau,\beta)$. By
(c)  of \ref{rho} we have $m(\tau,\beta)\leq\max(m(\tau,\alpha), m(\alpha,\beta))$
and so $m(\tau,\alpha)<m(\tau,\beta)\leq m(\alpha,\beta)$ follows. 

\noindent{\bf Case 3.} $\alpha<\tau<\beta$.

By  (d') and \ref{version4} we have that $m(\alpha,\tau)\leq m(\alpha, \beta)$.
 This completes the proof of the claim.

By the claim to obtain the required contradiction with our initial assumption it is enough to
find
$\theta<\kappa$ such that $A_\theta\subseteq M$ and (c), (d) are satisfied for
each $\xi\in A_\theta$. But $\zeta, \delta', \Delta$ are all elements of $M$, so
using the pairwise disjointness of the $A_\theta$s 
and so of the sets $B_\theta=\{a_\xi\setminus \Delta: \xi\in A_\theta\}$ it is easy to find in $M$
 a $\theta<\kappa$ satisfying 
$B_\theta\cap\mu_{\delta'}(\pi)=\emptyset$, 
$B_\theta\cap\bigcup\{\mu_{\zeta}(\tau): \tau\in \Delta\}=\emptyset$.
Then we also have $A_\theta, B_\theta\subseteq M$ as $\lambda\subseteq \delta\subseteq M$ and these
are sets of cardinalities not bigger than $\lambda$. 
But this guarantees (c) and (d) for
each $\xi\in A_\theta$, gives the required contradiction with the definition of $A_\theta$ and
completes the proof of the proposition.
\end{proof}

\vfill
\break
\begin{definition}[\cite{bs}]\label{deltaproperty} A
function $f:[\lambda^{++}]^2\rightarrow[\lambda^{++}]^{\leq\lambda}$ is said to have
 property $\Delta$ if and only if  
 whenever $\{a_\xi:\xi<\lambda^+\}$ is a collection of subsets of $\lambda^{++}$
of cardinalities $<\lambda$,
then there are $\xi,\xi'<\lambda^+$ satisfying the following $\Delta$-relations:
for any $\tau\in a_\xi\cap a_{\xi'}$,
$\alpha\in a_\xi- a_{\xi'}$, $\beta\in a_{\xi'}-a_{\xi}$ we have
\begin{enumerate}
\item $\tau<\alpha,\beta\ \Rightarrow\ \tau\in f(\alpha,\beta)$
\item $\beta>\tau\ \Rightarrow\ f(\alpha,\tau)\subseteq f(\alpha,\beta)$
\item $\alpha>\tau\ \Rightarrow\ f(\beta,\tau)\subseteq f(\alpha,\beta)$
\end{enumerate}
We say that  property $\Delta$ is  collectionwise if and only if under the above hypothesis
 there is $A\subseteq \lambda^+$ of cardinality $\lambda^+$ such that the
$\Delta$-relations are satisfied for all distinct $\xi,\xi'\in A$.

\end{definition}

\begin{theorem} Suppose that $\lambda^{<\lambda}=\lambda$ is a regular cardinal and  that
$\mu$ is  a  $(\lambda^+,\lambda^{++})$-cardinal . Then there is a function
$f:[\lambda^{++}]^2\rightarrow [\lambda^{++}]^{\leq\lambda}$
with  collectionwise property $\Delta$.
\end{theorem}
\begin{proof} Let $\kappa=\lambda^+$.
 Let $m$ be a $\mu$-coloring and for $\alpha<\kappa^+$ let
$(\mu_\xi(\alpha))_{\xi<\kappa}$ be the $\mu$-sequence at $\alpha$.
Let $\alpha<\beta$, and  put 
$$f(\alpha,\beta)=\mu_{m(\alpha,\beta)}(\alpha)=\{\xi<\alpha:\ m(\xi,\alpha)\leq m(\alpha,\beta)\}.$$

Find $A$ as in  \ref{velickovic}. Now suppose, that $\alpha,\beta,\tau$
are as in Definition \ref{deltaproperty}.  To prove (1) note that in this case
\ref{velickovic} gives that $m(\tau,\alpha),\ m(\tau,\beta)\leq m(\alpha,\beta)$.

By the symmetry, in the proof of (2) and (3) we may assume that $\alpha<\beta$.
We have two cases  $\tau<\alpha<\beta$ and $\alpha<\tau<\beta$. In the first
case using (1) and (2) of \ref{velickovic} and Lemma \ref{mucoloringcoherence} we get
$$f(\tau,\alpha)=\mu_{m(\tau,\alpha)}(\tau)=\mu_{m(\tau,\alpha)}(\alpha)\cap\tau
\subseteq \mu_{m(\alpha,\beta)}(\alpha)=f(\alpha,\beta).$$
$$f(\tau,\beta)=\mu_{m(\tau,\beta)}(\tau)=\mu_{m(\tau,\beta)}(\beta)\cap\tau
\subseteq \mu_{m(\alpha,\beta)}(\beta)\cap\alpha=\mu_{m(\alpha,\beta)}(\alpha)=f(\alpha,\beta).$$
In the second case using (1) of \ref{velickovic}  we get
$$f(\alpha, \tau)=\mu_{m(\alpha,\tau)}(\alpha)
\subseteq \mu_{m(\alpha,\beta)}(\alpha)=f(\alpha,\beta).$$
\end{proof}

\vskip 26pt

\section{Generic stepping-up}

Inductive constructions along $(\kappa,\kappa^+)$-cardinals as in Section 3,
gaps and nonreflection inherent in them as in Section 4 and coherent partitions of pairs
as in Section 5 can be unleashed in the context of constructions of forcing notions. 
We obtain stronger versions of  all these phenomena in the generic extension.
Often it is the only way of stepping up of the above phenomena from $\wp_\kappa(\kappa)$
to $\wp_\kappa(\kappa^+)$. This is related to the fact that
$2$-cardinals cohabit with GCH in the constructible universe, and GCH gives some Ramsey
property of cardinals in the form of nice cases of the Erdos-Rado theorem. Thus if we want to
get rid of both Ramsey charged principles as the Chang's conjecture and GCH we need to
force $2^{<\kappa}$ above $\kappa$.

In this section we consider only $(\omega_1,\omega_2)$-cardinals,
that is, subfamilies of $[\omega_2]^{\omega}$ because preserving $\omega_1$ 
is by far the most important cardinal preservation in the context of generic extensions.

Probably the earliest problem of constructing a forcing with a stepping
up tool was  of adding a Kurepa
tree by a forcing notion satisfying the c.c.c. known as Generic Kurepa Hypothesis. It was 
shown by Jensen (unpublished, see \cite{jensensz}) that $\square_{\omega_1}$ implies that
a Kurepa tree can be added by a c.c.c. forcing notion\footnote{In \cite{jensensz} it is shown that
a Mahlo cardinal is sufficient and necessary for 
obtaining the consistency of nonexistence of a 
c.c.c. forcing which adds a Kurepa tree. Note that
it is clear that PFA implies that there is no
c.c.c. forcing which adds a Kurepa tree; deciding 
the tree ordering in the tree would require 
meeting only $\omega_1$ dense sets, thus the 
Kurepa tree would exist in the universe, but 
PFA implies the negation of the weak Kurepa Hypothesis
(see \cite{baumgartnerhandbook}).
}.
In \cite{boban}, Velickovic constructed a c.c.c. forcing as above
using directly the $\rho$-function based on $\square_{\omega_1}$.
Recall from Section 4 that the existence of a $(\omega_1,\omega_2)$-cardinal
already implies the existence of a Kurepa tree.

Common stepping up tools hidden in $2$-cardinals and used for construction of  c.c.c. forcings are
 functions $f:[\omega_2]^2\rightarrow [\omega_2]^\omega$ 
or  $f:[\omega_2]^2\rightarrow \omega_1$. The reason they appear in the proofs
of the c.c.c. of forcing notions which add some interesting structures on $\omega_2$
 is that many structures define an associated function $F:[\omega_2]^2\rightarrow \omega_2$.
If the forcing is to be c.c.c. for every $F(\alpha,\beta)$ there must be a countable
set $A_{\alpha,\beta}$ in the ground model such that $F(\alpha,\beta)\subseteq A_{\alpha,\beta}$.
In other words
if our forcing allows uncountably many possible values of $F(\alpha,\beta)$ it is not c.c.c.
So the forcings for the results mentioned above
 usually have the form $\PP\ni p=(a_p, S(a_p))$  such that
\begin{itemize}
\item $a_p\in [\omega_2]^{<\omega}$,
\item $S(a_p)$ is some finite structure,
\item the behavior of  $S(a_p)$ is limited on the pairs of $a_p$ by $f$.
\end{itemize}

A prototypical example of adding the third limiting condition above to the first two is
considered by Baumgartner in \cite{baumgartneraml} where the consistency of the existence
of a family of size $\omega_2$ of uncountable subsets of $\omega_1$ with finite pairwise
intersections (strong almost disjoint family) is proved. Baumgartner first constructs a 
 collection of size $\omega_2$ of uncountable subsets of $\omega_1$ with countable pairwise
intersections, and then requires the finite approximations to the elements of
a generic strong almost disjoint family to be included in the elements of the collection.
This does the trick needed for the c.c.c. of the forcing with the finite approximations.

\noindent In \cite{bs}, J. Baumgartner and S. Shelah  solve an important and long standing 
problem concerning scattered compact spaces or superatomic
Boolean algebras,  first forcing  a function with 
$\Delta$-property and then using it to define a c.c.c. forcing which adds
the Boolean algebra.  As we have seen in Section 5 one can naturally obtain
a function with  property $\Delta$ using a $2$-cardinal.  The result says that
it is consistent that there is a superatomic Boolean algebra of countable width and 
height $\omega_2$.
In this seminal paper  $S(a_p)$ is roughly a finite Boolean algebra
generated by elements indexed by $\omega\times a_p$ and if  two of the generators
$g_{\alpha,n}$, $g_{\beta,k}$ are incomparable, then their meet is in the algebra generated by
the generators with indices in $\omega\times f(\alpha, \beta)$ (compare with the construction in
Section 3). 
This construction had several refinements and modifications  in various directions (\cite{rabus},
\cite{init}, \cite{christina})

A weaker version of a function with $\Delta$-property often used is the following:

\begin{definition}[\cite{stevomach}]
A function $f:[\omega_2]^2\rightarrow \omega_1$ is called
unbounded if and only if for every uncountable pairwise disjoint family $A\subseteq [\omega_2]^{<\omega}$
of finite subsets of $\omega_2$, for every $\delta\in \omega_1$ there are distinct $a, b\in A$
such that $f(\alpha,\beta)>\delta$ for every $\alpha\in a$ and every $\beta\in b$.
\end{definition}

It is straightforward, for example using  property $\Delta$ to prove that the $\mu$-coloring
of Section 5 for an $(\omega_1,\omega_2)$-coloring is an unbounded function.
The existence of such a function
is equivalent to the negation of Chang's Conjecture, as 
shown in \S 3 of  \cite{stevomach}. For more on unbounded functions see \cite{stevocoherent}
or \cite{unbounded}.
A function is used in \cite{stevomach} to show that under MA$_{\omega_2}$
 Chang's conjecture is equivalent to the partition relation that
says that every coloring of $\omega_2^2$ into $\omega$ colors 
is constant on the product of some two infinite sets. 
Unbounded functions were also  used by Martinez and Soukup to force superatomic Boolean algebras with
prescribed cardinal sequences (see \cite{soukupunb}).

 A similar application of an {\it unbounded}
 function is presented in \cite{fans}, where it is 
 shown that the failure of Chang's conjecture and
 MA$_{\omega_2}$ imply that the product 
 $S(\omega_2)\times S(\omega_2)\times\omega_1$ is normal,
 where $S(\omega_2)$ denotes the sequential fan with 
 $\omega_2$-many spines.

One could interpret some of the uses of morasses for generic stepping up in the spirit of
our Section 3. For exampe
Irrgang in \cite{irrgang} defines a forcing by recursion along a morass. In our terminology and
approach presented in Section 3, this corresponds to defining a family  of countable forcing notions
$(\PP_X: X\in\mu)$ where $\mu$ is a $2$-cardinal together with the appropriate embeddings and then
making sure that the limit along the directed set is a c.c.c. notion of forcing.

In some cases however it is impossible to obtain a required  consistency 
by building a c.c.c. forcing using a stepping-up structure which can be added 
by forcing preserving CH.

In papers \cite{unbounded}, \cite{hajnal} we considered forcing notions with side conditions
in $2$-cardinals (and $2$-semi cardinals - semimorasses of \cite{semi}). 
This is a version of Todorcevic's method of models as side conditions in the case
when one considers matrices of models and not just $\in$-chains of models (see \S 4 of \cite{stevoside}). 
The point was that many of the elementary properties of $2$-cardinals simplify life
if one works with the Todorcevic's method assuming moreover that the models $M$
which appear as side conditions satisfy $M\cap\omega_2\in \mu$ where
$\mu$ is a $2$-cardinal. For this one takes a stationary $2$-cardinal, actually
it is even better to take stationary coding sets because then we have \ref{stationarycoding}. 
The forcings assume the form
$\PP\ni p=(a_p, S(a_p), \mathcal F_p)$  such that
\begin{itemize}
\item $a_p\in [\omega_2]^{<\omega}$,
\item $S(a_p)$ is some finite structure, 
\item $\mathcal F_p\in [\mu]^{<\omega}$,
\item the behaviour of  $S(a_p)$ on pairs $\{\alpha,\beta\}\subseteq a_p$
is limited by every $X\in \mathcal F_p$ such that $\alpha,\beta\in X$.
\end{itemize}

For example in \cite{hajnal} the distance  $|\phi_\alpha(\gamma)-\phi_\beta(\gamma)|$ 
fore some $\gamma$ between  two generically constructed
functions $\phi_\alpha$ and $\phi_\beta$ in $S(a_p)$ is limited by sums of order types of 
appropriate elements of $\mathcal F_p$ of rank not bigger than $\beta$.
The result is the solution of a problem of Hajnal by proving the consistency of
the existence of  a well-ordered $\omega_2$-chain of functions in $\omega_1^{\omega_1}$
modulo finite sets. It is also shown that such a chain cannot be added by a c.c.c. forcing
over a model of CH.
This method was extensively analyzed in the context of morasses by Morgan in \cite{morgancrm}

Having in mind forcing with side conditions in $2$-cardinals one can revise the use
of stepping up tools for obtaining c.c.c. notions of forcing. Namely, instead of 
obtaining  complicated functions $f:[\omega_2]^2\rightarrow [\omega_2]^{\omega}$
and then defining forcing notions $\PP\ni p=(a_p, S(a_p), F_p)$ as described above, in particular
satisfying $F_p(\{\alpha,\beta\})\subseteq f(\{\alpha,\beta\})$ one can directly consider
a forcing notion $\QQ\ni q=(a_p, F_p, \mathcal F_p)$ where one requires
$F_p(\{\alpha,\beta\})\subseteq  X$ for every $X\in \mathcal F_p$ such that $\alpha,\beta\in X$.
This way one can force directly (without using property $\Delta$)  a superatomic algebra
of  Baumgartner and Shelah like in Section 3.3. of \cite{unbounded}.
Actually in \cite{christina} this route was taken and a stronger property than 
property $\Delta$ was obtained where   one requires
$a_\xi\cap\min \{\alpha,\beta\}\subseteq f(\alpha,\beta)$ instead of just $a_\xi\cap a_{\xi'}\cap
\min\{\alpha,\beta\}\subseteq f(\alpha,\beta)$ of (1) of Definition \ref{deltaproperty}.
It turned out that a function with such a property $\Delta$ cannot exist under CH unlike
the usual property  $\Delta$.

The final conclusions in \cite{christina} refer to topology  as well as Banach spaces and
Boolean algebras. For example we answer a question of Todorcevic from \cite{stevoirr}
showing that it is consistent that there are countably irredundant Boolean algebras 
of size $\omega_2$, or we obtain the first example of a Banach space of density $\omega_2$
without uncountable biorthogonal systems. The Banach space is of the form $C(K)$
where compact $K$ exhibits several new topological properties.

\section{Towards $n$-cardinals}

Although Jensen in his monumental work provided us with higher gap morasses (\cite{devlin2})
and although they can be simplified (\cite{vellemanhigher}, \cite{morganhigher}, \cite{szalkai})
and even some attempts of generic stepping up were made (\cite{irrgangjsl}),
one can safely claim that what we have at the moment is unsatisfactory, especially in the context of
basic questions concerning stepping up and gaps like whether it is consistent that there is a superatomic algebra
of  countable width and height $\omega_3$ (or higher) or whether it is consistent that there  is a Banach space
of density $\omega_3$ (or higher) without uncountable biorthogonal systems. 
On the other hand the level of complication of higher gap morasses in the context of  
the lack of spectacular  applications makes them remote  for most set-theorists.

One possible approach to $n$-cardinals as a structure where (oversimplifying) higher gap morass structure is
replaced by $\in$ and $\subseteq$ as in the case of $2$-cardinals could be to see  a $2$-cardinal as a 
pair of families of sets $\kappa^+=\{\alpha:\alpha<\kappa^+\}\subseteq \wp_{\kappa^+}(\kappa^+)$
 and a $(\kappa,\kappa^+)$-cardinal
$\mu\subseteq \wp_\kappa(\kappa^+)$. With this in mind one can define, say a $3$-cardinal as 
two (really three, together with $\kappa^{++}$) families:
\begin{itemize}
\item $\mu_1\subseteq \wp_{\kappa^+}(\kappa^{++})$ 
\item $\mu_2\subseteq \wp_{\kappa}(\kappa^{++})$
\end{itemize}
such that $\mu_1$ is a $(\kappa^+,\kappa^{++})$-cardinal,
such that $\mu_2$ is a $(\kappa,\kappa^{++})$-semicardinal (i.e. a neat 
$(\kappa,\kappa^{++})$-semimorass of \cite{semi}) and moreover 
$\mu_1$ and $\mu_2$ are bound by the following coherence condition
which steps up our coherence lemma \ref{coherence}:

Whenever $\alpha, \beta\in \kappa^{++}$, $X\in \mu_1$ of minimal
rank containing $\alpha,\beta$ and  $A, B\in \mu_2$ of the same rank containing $\alpha,\beta$,
then 
$$A\cap X\cap  \min\{\alpha,\beta\}=B\cap X\cap  \min\{\alpha,\beta\}.$$

Using forcing with side conditions one can prove the consistency of the existence of such objects,
however their usefulness is unclear. Also its relation to higher gap morasses is unclear and 
almost certainly the above structures are less powerful. It may also be possible that
already a gap two morass is too complicated to be comprehensibly expressed in terms of 
$\in$ and $\subseteq$.

A similar approach focused on the applications of stepping-up in building
forcing notions is taken by I. Neeman in \cite{neemantwo} or by
B. Velickovic and G. Venturi \cite{venturi} where forcing side conditions 
have two types of models, those which are countable and those which have cardinality $\omega_1$.

Perhaps for dealing with problems like those mentioned at the beginning of this section having
a transparent interaction among elementary submodels of several cardinalities like in
the coherence relation mentioned above could be helpful like it was helpful 
in \cite{hajnal} or \cite{unbounded}. However some surprising limitations are
certainly awaiting, for example Shelah showed in \cite{shelahhajnal} that unlike $\omega_2$-chains
in $\omega_1^{\omega_1}$ (\cite{hajnal}) modulo finite sets there cannot be  $\omega_4$-chains
in $\omega_3^{\omega_3}$ modulo finite sets. Some limitations concerning 
superatomic Boolean algebras are also well known (see \cite{martinezquestions}).

\bibliographystyle{amsplain}

\end{document}